\documentclass[11pt,a4paper]{article}

\usepackage[english]{babel}
\usepackage{amsmath}
\usepackage{amsfonts,amsthm,amssymb}
\usepackage{blkarray}
\usepackage{latexsym,bm}
\usepackage{indentfirst}
\usepackage{graphicx}
\usepackage{color}
\usepackage{tcolorbox}
\usepackage{tikz}
\usepackage{caption}
\usepackage{multicol}
\usepackage{multirow}
\usepackage[letterpaper,top=2cm,bottom=2cm,left=3cm,right=3cm,marginparwidth=1.75cm]{geometry}
\usepackage[colorlinks=true,anchorcolor=blue,filecolor=blue,linkcolor=blue,urlcolor=blue,citecolor=red]{hyperref}

\newtheorem{theorem}{Theorem}[section]

\newtheorem{lemma}{Lemma}[section]
\newtheorem{false statement}{False statement}
\newtheorem{corollary}{Corollary}[section]

\theoremstyle{definition}

\newtheorem{claim}{Claim}

\newtheorem{remark}{Remark}

\numberwithin{case}{subsection}

\title{\textbf{Extremal results on the second largest eigenvalue of graphs with given order}
\thanks{
This work is supported by National Natural Science Foundation of China (Nos. 12171154, 12301438), Chenguang Program of Shanghai Education Development Foundation and Shanghai Municipal Education Commission (No. 23CGA37).}
}
\author{
	Zhiwen Wang,\,\,
	Ji-Ming Guo\thanks{Corresponding author.\\ 
	\null\hspace{6.3mm}Email addresses:  walkerwzw@163.com (Z. Wang); jimingguo@hotmail.com (J.-M. Guo).}\vspace{4mm}\\
	School of Mathematics, East China University of Science and Technology,\\ Shanghai, 200237, P. R. China.
}
\date{\null}
\begin{document}
\maketitle

\begin{abstract}
Adding or removing an edge from a connected graph will strictly increase or decrease its spectral radius, respectively. 
This phenomenon, however, fails for the second largest eigenvalue.
In this paper, we demonstrate the effects on the second largest eigenvalue $\lambda_2(G)$ of a connected graph $G$ after edge addition or deletion. 

In 1989, Chung, Graham and Wilson showed $\max\{|\lambda_2|,|\lambda_n|\}>\Omega(n)$ for dense $K_{r+1}$-free graphs of order $n$, giving spectral comprehension of existence of large clique or independent set, respect to Ramsey theory.
Applying the results of effects on $\lambda_2$ after edge operations, we determine the maximum value of $\lambda_2$ among all $K_{r+1}$-free connected graphs with given order, and completely characterize the extremal graphs.

Moreover, for arbitrary given graph $F$, we investigates the maximum second largest $\lambda_2(G)$ among $F$-free connected graphs of order $n$.
Let $\rho^*(n,F)$ be the maximum spectral radius of $F$-free graphs on $n\ge n_F$ vertices, and $G^*(n,F)$ be a graph with its spectral radius $\rho\big(G^*(n,F)\big)=\rho^*(n,F)$.
We prove that, for an $F$-free connected graph $G$ of order $n\ge f(n_F)$, \\(1) if $n$ is odd, then $$\lambda_2(G)\le\rho^*\left(\frac{n-1}{2},F\right)$$ with equality if and only if $G\in \mathcal{I}\big(G^*(\frac{n-1}{2},F),G^*(\frac{n-1}{2},F)\big)$; and\\ (2) if $n$ is even, and $F$ does not contain cut edges, then the graph $G^\dag$ with the maximum second largest eigenvalue satisfies $$\lambda_2(G^\dag)=\rho^*\left(\frac{n}{2},F\right)-o(1)$$ and $G^\dag\in \mathcal{E}\big(H_1,H_2\big)$, where $H_1$ and $H_2$ are $F$-saturated graphs on $\frac{n}{2}$ vertices.

This allows us to explore the maximum value of $\lambda_2$ for valuable families of various $F$-free graphs, such as $F$ is a color critical graph or a minor.
In particular, other than a complete graph $K_{r+1}$, when $F$ is a book graph $B_{k+1}$ or an odd cycle $C_{2k+1}$, we are able to determine the maximum second largest eigenvalue for $F$-free connected graphs of given order, and completely characterize the extremal graphs.
We also describe the maximum second largest eigenvalue for planar graphs and outerplanar graphs.
\end{abstract}

\section{Introduction}\label{section:1}
The Ramsey number $r(s,t)$ is the smallest positive integer $n$ such that any graph $G$ of order $n$ contains a complete graph $K_s$ or a null graph $\overline{K_t}$ as an induced subgraph.
A landmark work by Ramsey \cite{R30} showed $r(s,t)$ is finite for each $r$ and $s$.
Bounding the value of $r(s,t)$ contributes to Erd\H{o}s and Szekeres \cite{ES35}, who seminally proved a non-trivial upper bound that $r(s,t)\le\left(\begin{matrix}
     s+t-2\\
     t-1
\end{matrix}\right)$, and Erd\H{o}s \cite{E47}, providing a lower bound for sufficiently large $s$ and $t$.
This shows that a graph of large order will contains either a large clique or a large independent set.

Given a graph $F$, a graph $G$ is called $F$-free if $G$ does not contain a copy of $F$ as a subgraph.
Denoted by $\lambda_1(G)\ge\lambda_2(G)\ge\cdots\ge\lambda_n(G)$ in non-increasing order the eigenvalues of adjacency matrix of the graph $G$.
The symbol $G$ may be omitted if no ambiguity.

In 1989, Chung, Graham and Wilson \cite{CGW89} presented a version of spectral theory on Ramsey number, showing that
a dense $K_{r}$-free graph of order $n$ and size $m\ge cn^2$, where $c\in (0,\frac{1}{2})$, satisfies that either $\lambda_2>c'n$ or $\lambda_n<c'n$, where $c'=c'(r,c)$ is a constant and $n$ is sufficiently large.
It is clear that there exists an independent set of large size in a sparse graph.
This work of Chung, Graham and Wilson implies that when $$\lambda:=\max\{|\lambda_2|,|\lambda_n|\}\le\Omega(n),$$ the graph will contain either a complete graph $K_{r'}$ or a null graph $\overline{K_{s'}}$, where $r'=r'(n,\lambda)$ and $s'=s'(n,\lambda)$.
In 2004, Bollob\'as and Nikiforov \cite{BN04} gave a direct refinement by using eigenvalue interlacing techniques, proving that a dense $K_r$-free graph always holds $\lambda_n\le -\alpha n$, where $\alpha$ is independent of $n$. 
A subsequent work by Nikiforov \cite{N06} further determined a more precise bound on $\lambda_n$, that
$\lambda_n\le -\frac{2^{r+1}m^r}{rn^{2r-1}}$ for a 
$K_{r+1}$-free graph of order $n$ and size $m$.
This motivates us to investigate how large $\lambda_2$ can be in $K_r$-free graphs of order $n$.

Let $T_{n,r}$ be the $r$-partite Tur\'an graph of order $n$, i.e., a complete multipartite graph in which each color set is of size $\lfloor\frac{n}{r}\rfloor$ or $\lceil\frac{n}{r}\rceil$. 
The largest eigenvalue $\lambda_1$ is also called the spectral radius of $G$, denoted by $\rho(G)$.
In Theorem \ref{thm:complete graph-2}, we give a complete answer to the above question, that is, among $K_{r+1}$-free connected graphs $G$ of order $n$, we prove
\begin{align*}
    \lambda_2(G)\le\begin{cases}
     \rho(T_{\frac{n-1}{2},r}),&\ \mathrm{if}\ n\ \mathrm{is} \ \mathrm{odd};\\
     \lambda_2(T_{\frac{n}{2},r}uvT_{\frac{n}{2},r}),&\ \mathrm{if}\ n\ \mathrm{is} \ \mathrm{even},
     \end{cases}
\end{align*}
where $T_{\frac{n}{2},r}uvT_{\frac{n}{2},r}$ is obtained from two copies of $T_{\frac{n}{2},r}$ by embedding an new edge between their vertices $u$ and $v$ of minimum degree.
And the family of graphs attaining the upper bounds is characterized completely. 
This means $\lambda_2(G)\le \left(1-\frac{1}{r}\right)\frac{n}{2}$, giving an estimate $c'<\frac{1}{2}-\frac{1}{2r}$ when $\lambda_2>c'n$ holds in \cite{CGW89}.
Meanwhile, this work is a generalization of known upper bounds on trees \cite{N82,S95} and bipartite graphs \cite{P88, ZLW12}, since a graph having $K_{r+1}$ as a subgraph is of chromatic number at least $r+1$.

In fact, it is not satisfied with merely giving the maximum value of $\lambda_2$ in $K_{r+1}$-free graphs.
We tend to find sharp upper bounds of $\lambda_2$ in general $F$-free graphs with given order.
This research direction also is extension of spectral counterpart of extremal theory about the spectral radius $\lambda_1$.

For decades, the study on spectral radius of graphs has received great attention.
The extremal number $\textrm{ex}(n, F)$, known as Tur\'an number, is the maximum number of edges in an $n$-vertex graph not containing $F$.
Due to the relation $\rho(G)\ge \frac{2m}{n}$ for a graph $G$ of order $n$ and size $m$ with its spectral radius $\rho(G)$,
results on maximizing the spectral radius $\rho(G)$ for $F$-free graphs $G$ could give an upper bound to $\textrm{ex}(n,F)$, that is,
$\textrm{ex}(n,F)\le \frac{n\rho^*}{2},$
provided that $\rho^*$ is the maximum $\rho(G)$ among all $F$-free graphs $G$ on $n$ vertices.
Much meaningfully, when the extremal graph with spectral radius $\rho^*$ is regular, the upper bound is sharp. So it can be viewed as an improvement of spectral version to extremal result on $\textrm{ex}(n,F)$.
One may refer to much nice literature for various $F$-free graphs, such as complete graphs $K_{r+1}$ \cite{BN07,N02,N07,W86}, complete bipartite graphs $K
_{s,t}$ \cite{BG09,N07}, cycles $C_k$ \cite{LN23,N08,NP20, ZL22} and graph-minors \cite{N17,T19, TT17, ZL22} and the references therein.

For two disjoint graphs $G$ and $H$, let $\mathcal{I}(G,H)$ be the family of connected graphs obtained from $G$ and $H$ by connecting a new vertex to some vertices of each $G$ and $H$.
And let $\mathcal{E}(G,H)$ be the family of graphs obtained from $G$ and $H$ by connecting $G$ and $H$ with a new edge.
In particular, if $u\in V(G)$ and $v\in V(H)$, denote by $GuvH$ the graph in $\mathcal{E}(G,H)$ when the added edge is $uv$.

Let $\rho^*(n,F)$ be the maximum spectral radius of $F$-free graphs on $n\ge n_F$ vertices, and $G^*(n,F)$ be an extremal graph with $\rho\big(G^*(n,F)\big)=\rho^*(n,F)$. 
We mark that $n$ is sufficiency large if $n_F$ is $\infty$.
A graph $G$ is $F$-saturated if $G$ is $F$-free but $G+e$ for any $e\notin E(G)$ contains a copy of $F$ as a subgraph. We extend the spectral version of extremal results on $\lambda_1$ to the second eigenvalue $\lambda_2$.
For an $F$-free connected graph $G$ of order $n\ge f(n_F)$, in Theorems \ref{thm:F-free and odd n} and \ref{thm:F-free and even n}, we prove that, when $n$ is odd,
$$\lambda_2(G)\le\rho^*\left(\frac{n-1}{2},F\right)$$
with the equality holds if and only if $G\in \mathcal{I}\big(G^*(\frac{n-1}{2},F),G^*(\frac{n-1}{2},F)\big)$; and when $n$ is even, and $F$ does not contain cut edges (this restriction is not absolute),
$$\lambda_2(G^\dag)=\rho^*\left(\frac{n}{2},F\right)-o(1),$$
where $G^\dag$ is a graph with the maximum value of $\lambda_2$.

In particular, let $F$ be a book graph $B_{r+1}$, i.e., $r+1$ triangles by sharing an edge, or an odd cycle $C_{2k+1}$ ($k\ge 1$), then
\begin{align*}
    \lambda_2(G)\le\begin{cases}
     \rho(T_{\frac{n-1}{2},2}),&\ \mathrm{if}\ n\ \mathrm{is} \ \mathrm{odd};\\
     \lambda_2(T_{\frac{n}{2},2}uvT_{\frac{n}{2},2}),&\ \mathrm{if}\ n\ \mathrm{is} \ \mathrm{even},
     \end{cases}
\end{align*}
and the graph attaining the upper bound belongs to the family $\mathcal{I}(T_{\frac{n-1}{2},2},T_{\frac{n-1}{2},2})$ for odd $n$, and is isomorphic to $T_{\frac{n}{2},2}uvT_{\frac{n}{2},2}$ for even $n$.
Recently, Brooks, Gu, Hyatt, Linz and Lu \cite{BGHLL25} investigated outerplanar graphs, and obtained the maximum value of $\lambda_2$ as well as the extremal graphs, when the order is sufficiently large.
It is noted that a graph is outerplanar if and only if it not contains a member of $\{K_4,K_{2,3}\}$ as a minor.
We get the upper bound of $\lambda_2$ for $K_r$-minor-free or $K_{s,t}$-minor-free graphs.
This also allow us to determine the maximum value of $\lambda_2$ for planar graphs, or outerplanar graphs (showing the result by Brooks, Gu, Hyatt, Linz and Lu \cite{BGHLL25} holds for any odd order and providing a tight maximum value for even order).

The rest of the paper is organized as follows.
In Section \ref{section:2}, we establish relationships of the second largest eigenvalues between a graph and the resulting graph after edge addition or deletion. 
And we introduce some critical lemmas in order to bound $\lambda_2$.
In Section \ref{section:3}, we first prove our main results (Theorems \ref{thm:F-free and odd n} and \ref{thm:F-free and even n}), giving the maximum value of $\lambda_2$ among $F$-free graphs, where $F$ is any graph not containing cut edges.
Then we apply these results to some special $F$-free graphs, such as (1) color critical graphs, including complete graphs, Book graphs and odd cycles, and (2) graph-minors, including $K_r$-minor and $K_{s,t}$-minor. 

\section{Preliminaries}\label{section:2}

The following theorem is the well known Courant-Fischer Theorem (see \cite{HJ85} for example).

\begin{theorem}\label{thm:C-F thm}[Courant-Fischer Theorem]
Let $M$ be a real and symmetric matrix with eigenvalues $\lambda_1\ge \lambda_2\ge\cdots\ge \lambda_n$. Then
\begin{equation*}
\lambda_k=\max_{\substack{\mathbb{S}\subseteq \mathbb{R}^n\\ \dim\mathbb{S}=k}}\min_{\textit{\textbf{x}}\in\mathbb{S}}\frac{\textit{\textbf{x}}^TM\textit{\textbf{x}}}{\textit{\textbf{x}}^T\textit{\textbf{x}}}=\min_{\substack{\mathbb{S}\subseteq \mathbb{R}^n\\ \dim\mathbb{S}=n-k+1}}\max_{\textit{\textbf{x}}\in\mathbb{S}}\frac{\textit{\textbf{x}}^TM\textit{\textbf{x}}}{\textit{\textbf{x}}^T\textit{\textbf{x}}}.
\end{equation*}
\end{theorem}

The Perron-Frobenius theorem states there exists a positive eigenvalue corresponding to $\lambda_1(G)$ if $G$ is connected (which means $A(G)$ is irreducible).

When we fix a space $\mathbb{W}$ of dimension $2$, by Courant-Fischer Theorem, we know that
\begin{align}\label{eq:relation of lambda 2}
\lambda_2=\max_{\substack{\mathbb{S}\subseteq \mathbb{R}^n\\ \dim\mathbb{S}=2}}\min_{x\in\mathbb{S}}\frac{\emph{\textbf{x}}^TA(G)\emph{\textbf{x}}}{\emph{\textbf{x}}^T\emph{\textbf{x}}}\ge \min_{\emph{\textbf{x}}\in\mathbb{W}}\frac{\emph{\textbf{x}}^TA(G)\emph{\textbf{x}}}{\emph{\textbf{x}}^T\emph{\textbf{x}}}.
\end{align}

By Perron-Frobenious theorem, the largest eigenvalue $\lambda_1$ of $G$ has a non-negative or non-positive eigenvector $\emph{\textbf{y}}$.
Let $\textit{\textbf{z}}$ be an eigenvector of $G$ corresponding to $\lambda_2$.
Then the elements of $\textit{\textbf{z}}$ are possibly positive, negative and zero.
We note the fact that $$\lambda_2=\min_{\emph{\textbf{x}}\in \mathrm{span}{\{\emph{\textbf{y}},\emph{\textbf{z}}}\}}\frac{\emph{\textbf{x}}^TA(G)\emph{\textbf{x}}}{\emph{\textbf{x}}^T\emph{\textbf{x}}}.$$
Denote the sets of vertices corresponding to positive elements, negative elements and zero elements of $\textit{\textbf{z}}$, by $V_+$, $V_-$ and $V_0$, respectively.

Next, we present the effects on the second largest eigenvalue of $G$ by graph operations.

\begin{theorem}\label{thm: add edge in N+/N-}
Suppose that $u,v\in V_+$ or $u,v\in V_-$, and $uv$ is not an edge in $G$. Then $\lambda_2(G+uv)\ge \lambda_2(G)$.
In particular, if $G$ is connected, then $\lambda_2(G+uv)>\lambda_2(G)$.
\end{theorem}
\begin{proof}
Let $\emph{\textbf{y}}$ and $\emph{\textbf{z}}$ be unit eigenvectors of $G$ with respect to $\lambda_1(G)$ and $\lambda_2(G)$, respectively. That is,
\begin{center}
$\lambda_1(G)=\emph{\textbf{y}}^TA(G)\emph{\textbf{y}}$\ \ and
$\lambda_2(G)=\emph{\textbf{z}}^TA(G)\emph{\textbf{z}}$.
\end{center}
We note $\emph{\textbf{y}}^T\emph{\textbf{z}}=0$.
Let $\mathbb{W}=\mathrm{span}{\{\emph{\textbf{y}},\emph{\textbf{z}}}\}$, then for any unit vector $\emph{\textbf{x}}\in \mathbb{W}$, we have $\emph{\textbf{x}}=k_1\emph{\textbf{y}}+k_2\emph{\textbf{z}}$ with $k_1^2+k_2^2=1$.


We first consider the case that $u,v\in V_+$. 
Then
\begin{align}
    \emph{\textbf{x}}^TA(G+uv)\emph{\textbf{x}}&=(k_1\emph{\textbf{y}}+k_2\emph{\textbf{z}})^TA(G+uv)(k_1\emph{\textbf{y}}+k_2\emph{\textbf{z}})\nonumber\\
    &=k_1^2\emph{\textbf{y}}^TA(G)\emph{\textbf{y}}+k_2^2\emph{\textbf{z}}^TA(G)\emph{\textbf{z}}+2k_1^2\emph{\textbf{y}}_u\emph{\textbf{y}}_v+2k_2^2\emph{\textbf{z}}_u\emph{\textbf{z}}_v+2k_1k_2(\emph{\textbf{y}}_u\emph{\textbf{z}}_v+\emph{\textbf{y}}_v\emph{\textbf{z}}_u)\label{eq:sign of a term}\\
    &\ge k_1^2\emph{\textbf{y}}^TA(G)\emph{\textbf{y}}+k_2^2\emph{\textbf{z}}^TA(G)\emph{\textbf{z}}\label{eq:connected-1}\\
    &=k_1^2\lambda_1(G)+k_2^2\lambda_2(G)\nonumber\\
    &\ge\lambda_2(G).\label{eq:connected-2}
\end{align}
Thus, by Eq.\,(\ref{eq:relation of lambda 2}),
$$\lambda_2(G+uv)\ge\min_{\emph{\textbf{x}}\in\mathbb{W}}\frac{\emph{\textbf{x}}^TA(G+uv)\emph{\textbf{x}}}{\emph{\textbf{x}}^T\emph{\textbf{x}}}\ge \lambda_2(G).$$

For the case that $u,v\in V_-$, the process is similar to the above, except determining the sign of $2k_1k_2(\emph{\textbf{y}}_u\emph{\textbf{z}}_v+\emph{\textbf{y}}_v\emph{\textbf{z}}_u)$ in Eq.\,(\ref{eq:sign of a term}).
If $2k_1k_2(\emph{\textbf{y}}_u\emph{\textbf{z}}_v+\emph{\textbf{y}}_v\emph{\textbf{z}}_u)\ge 0$, there is nothing to prove. 
If $2k_1k_2(\emph{\textbf{y}}_u\emph{\textbf{z}}_v+\emph{\textbf{y}}_v\emph{\textbf{z}}_u)<0$, we replace $\emph{\textbf{y}}$ with $-\emph{\textbf{y}}$, and then we get the proof of $\lambda_2(G+uv)\ge \lambda_2(G)$.

Suppose that the graph $G$ is connected. 
Then the equality $\lambda_2(G+uv)= \lambda_2(G)$ indicates that all inequalities in Eqs.\,(\ref{eq:connected-1}) and (\ref{eq:connected-2}) turn into equalities.
We know $\lambda_1(G)>\lambda_2(G)$ from Perron-Frobenius theorem. 
So $k_1=0$ and $2k_2^2\emph{\textbf{z}}_u\emph{\textbf{z}}_v=0$. 
Mention that $k_1^2+k_2^2=1$.
We have $k_2=1$, which deduces $\emph{\textbf{z}}_u\emph{\textbf{z}}_v=0$.
This is a contradiction to the assumption $u,v\in V_+$ or $V_-$.
Therefore, we have $\lambda_2(G+uv)>\lambda_2(G)$.

This completes the proof.
\end{proof}

The following two well-known conclusions may refer to literature (see \cite{CRS10} for example).
\begin{lemma}\cite{CRS10}\label{lm:subgraph}
Let $H$ be a subgraph of $G$, then $\lambda_1(G)\ge \lambda_1(H)$.
\end{lemma}

\begin{lemma}\cite{CRS10}\label{lm:Cauchy interlacing theorem}
Let $H$ be an induced subgraph of a graph $G$.
Then, for $1\le i\le |V(H)|$,
$$\lambda_i(G)\ge\lambda_i(H)\ge\lambda
_{i+|V(G)|-|V(H)|}(G).$$
\end{lemma}

\begin{theorem}\label{thm:remove edge between N+ and N-}
Suppose that $u\in V_+$, $v\in V_-$ and $uv$ is an edge in $G$. Then $\lambda_2(G-uv)> \lambda_2(G)$.
\end{theorem}
\begin{proof}
Let $\emph{\textbf{y}}$ be a unit eigenvector of $G-uv$ with respect to $\lambda_1(G-uv)$, and $\emph{\textbf{z}}$ be a unit eigenvector of $G$ with respect to $\lambda_2(G)$.
Then we have
\begin{equation}\label{eq:combi}
\begin{cases}
&\emph{\textbf{y}}^TA(G-uv)\emph{\textbf{y}}=\lambda_1(G-uv)\emph{\textbf{y}}^T\emph{\textbf{y}}=\lambda_1(G-uv),\\
&\emph{\textbf{y}}^TA(G-uv)\emph{\textbf{z}}=\lambda_1(G-uv)\emph{\textbf{y}}^T\emph{\textbf{z}},\\
&\emph{\textbf{z}}^TA(G-uv)\emph{\textbf{z}}=\emph{\textbf{z}}^TA(G)\emph{\textbf{z}}-2\emph{\textbf{z}}_u\emph{\textbf{z}}_v=\lambda_2(G)-2\emph{\textbf{z}}_u\emph{\textbf{z}}_v>\lambda_2(G).
\end{cases}
\end{equation}

Let $\mathbb{W}=\mathrm{span}{\{\emph{\textbf{y}},\emph{\textbf{z}}}\}$. We claim that $\dim(\mathbb{W})=2$.
On the contrary, suppose that $\emph{\textbf{y}}$ and $\emph{\textbf{z}}$ are collinear.
Then $\emph{\textbf{z}}$ is a unit eigenvector of $G-uv$ with respect to $\lambda_1(G-uv)$.
By Eq.\,(\ref{eq:combi}) we have
$$\lambda_2(G)<\emph{\textbf{z}}^TA(G-uv)\emph{\textbf{z}}=\lambda_1(G-uv).$$
It is clear that $A(G)=A(G-uv)+\textbf{e}_u\textbf{e}^T_v+\textbf{e}_v\textbf{e}^T_u$, where $\textbf{e}_u$ and $\textbf{e}_v$ are normal vectors with $u$-component one and $v$-component one, respectively.
Then
$A(G)\emph{\textbf{z}}=A(G-uv)\emph{\textbf{z}}+\left(\textbf{e}_u\textbf{e}^T_v+\textbf{e}_v\textbf{e}^T_u\right)\emph{\textbf{z}}$, which indicates
$$\left(\lambda_2(G)-\lambda_1(G-uv)\right)\emph{\textbf{z}}=\left[\textbf{0}^T\ \ \emph{\textbf{z}}_v\ \ \textbf{0}^T\ \ \emph{\textbf{z}}_u\ \ \textbf{0}^T\right]^T.$$
This gives that $\emph{\textbf{z}}_w=0$ for any vertex $w\in V(G)\setminus\{u,v\}$.
Since $u$ and $v$ is not an edge in $G-uv$. 
From $A(G-uv)\emph{\textbf{z}}=\lambda_1(G-uv)\emph{\textbf{z}}$, we get $\emph{\textbf{z}}_u=\emph{\textbf{z}}_v=0$.
So $\emph{\textbf{z}}=\textbf{0}$, a contradiction.
Thus $\dim(\mathbb{W})=2$.

By Lemmas \ref{lm:subgraph} and \ref{lm:Cauchy interlacing theorem}, we have $\lambda_1(G-uv)\ge\lambda_1(G-u)\ge\lambda_2(G)$.
For any nonzero vector $\emph{\textbf{x}}\in \mathbb{W}$, we have $\emph{\textbf{x}}=k_1\emph{\textbf{y}}+k_2\emph{\textbf{z}}$ and
   $$\emph{\textbf{x}}^T\emph{\textbf{x}}=k_1^2\emph{\textbf{y}}^T\emph{\textbf{y}}+k_2^2\emph{\textbf{z}}^T\emph{\textbf{z}}+2k_1k_2\emph{\textbf{y}}^T\emph{\textbf{z}}=k_1^2+k_2^2+2k_1k_2\emph{\textbf{y}}^T\emph{\textbf{z}}.$$
And we may let $k_1k_2\emph{\textbf{y}}^T\emph{\textbf{z}}\ge 0$ by replacing $\emph{\textbf{y}}$ as $-\emph{\textbf{y}}$ if $k_1k_2\emph{\textbf{y}}^T\emph{\textbf{z}}\le 0$.
Then, by Eq.\,(\ref{eq:combi}), we obtain
\begin{align}
    \emph{\textbf{x}}^TA(G-uv)\emph{\textbf{x}}&=(k_1\emph{\textbf{y}}+k_2\emph{\textbf{z}})^TA(G-uv)(k_1\emph{\textbf{y}}+k_2\emph{\textbf{z}})\nonumber\\
    &=(k_1\emph{\textbf{y}}+k_2\emph{\textbf{z}})^T\big(k_1\lambda_1(G-uv)\emph{\textbf{y}}+k_2A(G-uv)\emph{\textbf{z}}\big)\nonumber\\
    &=k_1^2\lambda_1(G-uv)+k_2^2\emph{\textbf{z}}^TA(G-uv)\emph{\textbf{z}}+2k_1k_2\lambda_1(G-uv)\emph{\textbf{y}}^T\emph{\textbf{z}}\nonumber\\
    &\ge k_1^2\lambda_1(G-uv)+k_2^2\lambda_2(G)+2k_1k_2\lambda_1(G-uv)\emph{\textbf{y}}^T\emph{\textbf{z}}\label{eq:+-1}\\
    &\ge (k_1^2+k_2^2)\lambda_2(G)+2k_1k_2\lambda_2(G)\emph{\textbf{y}}^T\emph{\textbf{z}}\label{+-2}\\
    &=\left(k_1^2+k_2^2+2k_1k_2\emph{\textbf{y}}^T\emph{\textbf{z}}\right)\lambda_2(G).\nonumber
\end{align}
This gives that
$$\frac{\emph{\textbf{x}}^TA(G-uv)\emph{\textbf{x}}}{\emph{\textbf{x}}^T\emph{\textbf{x}}}\ge \lambda_2(G).$$
Thus, by Eq.\,(\ref{eq:relation of lambda 2}),
$$\lambda_2(G-uv)\ge\min_{\emph{\textbf{x}}\in\mathbb{W}}\frac{\emph{\textbf{x}}^TA(G-uv)\emph{\textbf{x}}}{\emph{\textbf{x}}^T\emph{\textbf{x}}}\ge \lambda_2(G).$$

Now assume that $\lambda_2(G-uv)=\lambda_2(G)$.
There exists a nonzero vector $\emph{\textbf{x}}$ in $\mathbb{W}$, wrote as $k_1\emph{\textbf{y}}+k_2\emph{\textbf{z}}$, such that the above inequalities turns into equalities.
From Eqs.\,(\ref{eq:combi}) and (\ref{eq:+-1}), we get that $k_2=0$. Then from Eq.\,(\ref{+-2}), we know $\lambda_1(G-uv)=\lambda_2(G)$.
However, according to Theorem \ref{thm:C-F thm} and Eq.\,(\ref{eq:combi}), we have
$$\lambda_1(G-uv)=\max_{\emph{\textbf{w}}\in\mathbb{R}^n}\frac{\emph{\textbf{w}}^TA(G-uv)\emph{\textbf{w}}}{\emph{\textbf{w}}^T\emph{\textbf{w}}}\ge \emph{\textbf{z}}^TA(G-uv)\emph{\textbf{z}}>\lambda_2(G),$$
a contradiction. Thus we get $\lambda_2(G-uv)>\lambda_2(G)$, completing the proof. 
\end{proof}


\begin{lemma}\label{lem:null cut-set}
Let $\textbf{x}$ be an eigenvector of a graph $G$ corresponding to $\lambda_2(G)$.
Let $N_0$ be a vertex cut and $\textbf{x}\big|_{ {N_0}}=\emph{\textbf{0}}$. 
If $H$ is a connected component of $G-N_0$ and $\textbf{x}\big|_{ {V(H)}}\not=\emph{\textbf{0}}$,
then $\lambda_2(G)$ is an eigenvalue of $H$.
In particular, if $\textbf{x}\big|_{V(H)}>\emph{\textbf{0}}$ or $\textbf{x}\big|_{V(H)}<\emph{\textbf{0}}$, then $\lambda_2(G)$ is the largest eigenvalue of $H$.
\end{lemma}
\begin{proof}
Write $N_1$ for $V(H)$ and $N_2$ for $V(G)\backslash(N_0\cup N_1)$.
Under the partition $N_0\cup N_1\cup N_2$, we have the system of eigenequations
\begin{align*}
    \left[\begin{matrix}
    A(H) & A[N_1,N_0] & \textbf{O}\\
    A[N_0,N_1] & A[N_0] & A[N_0,N_2]\\
    \textbf{O} & A[N_2,N_0] & A[N_2]
    \end{matrix}\right]\cdot
    \left[\begin{matrix}
    \emph{\textbf{x}}\big|_{N_1}\\
    \emph{\textbf{x}}\big|_{N_0}\\
    \emph{\textbf{x}}\big|_{N_2}
    \end{matrix}\right]=
    \lambda_2(G)\left[\begin{matrix}
    \emph{\textbf{x}}\big|_{N_1}\\
    \emph{\textbf{x}}\big|_{N_0}\\
    \emph{\textbf{x}}\big|_{N_2}
    \end{matrix}\right],
\end{align*}
which implies by noting $\emph{\textbf{x}}\big|_{N_0}=\textbf{0}$
that
\begin{center}
    $A(H)\emph{\textbf{x}}\big|_{N_1}=\lambda_2(G)\emph{\textbf{x}}\big|_{N_1}$
\end{center}
So $\lambda_2(G)$ is an eigenvalue of $H$ since $\emph{\textbf{x}}\big|_{ {V(H)}}\not=\textbf{0}$.

If $\emph{\textbf{x}}\big|_{ {V(H)}}>\textbf{0}$ or $\emph{\textbf{x}}\big|_{ {V(H)}}<\textbf{0}$, then  $\lambda_2(G)$ is the largest eigenvalue of $H$ from Perron-Frobenius theorem.
\end{proof}

\begin{lemma}\label{lem:connected by negative edges}
Let $\textbf{x}$ be an eigenvector of a graph $G$ corresponding to $\lambda_2(G)$ and $H$ be an induced subgraph of $G$. 
If $\textbf{x}\big|_{V(H)}>\emph{\textbf{0}}$ and $\textbf{x}_w\le 0$ for any vertex $w\in V(G)\backslash V(H)$ adjacent to some vertex in $V(H)$, then $\lambda_2(G)\le \lambda_1(H)$.
The inequality is strict if $\textbf{x}_w< 0$ for some above vertex $w$.
\end{lemma}
\begin{proof}
Since $A(G)\emph{\textbf{x}}=\lambda_2(G)\emph{\textbf{x}}$, 
we have
$$\lambda_2(G)\cdot\emph{\textbf{x}}\big|_{V(H)}=A(H)\cdot\emph{\textbf{x}}\big|_{V(H)}+\epsilon\le A(H)\cdot\emph{\textbf{x}}\big|_{V(H)},$$
where $\epsilon$ is a vector of order $|V(H)|$ and $\epsilon\le\textbf{0}$. 
From Theorem \ref{thm:C-F thm}, this follows that
$$\lambda_2(G)\le \frac{\emph{\textbf{x}}^T\big|_{V(H)}A(H)\emph{\textbf{x}}\big|_{V(H)}}{\emph{\textbf{x}}^T\big|_{V(H)}\emph{\textbf{x}}\big|_{V(H)}}\le \lambda_1(H).$$
And it is easy to note $\lambda_2(G)< \lambda_1(H)$ if $\epsilon\not=\textbf{0}$.
\end{proof}



\section{Extremal results}\label{section:3}

For two disjoint vertex subsets $U$ and $W$ of a graph, we denote by $E(U,W)$ the set of edges between $U$ and $W$, and let $e(U,W)=|E(U,W)|$.
For a vertex subset $U$ of a graph, we denote by $E(U)$ the set of edges whose end-vertices belong to $U$, and let $e(U)=|E(U)|$.

For a graph $F$, a graph $G$ is called $F$-free if $G$ not contains a copy of $F$ as a subgraph.

The researches on the $k$th-largest eigenvalue are much less than those on the spectral radius.
As mentioned in Section \ref{section:1},
Chung et al. \cite{CGW89} showed either $\lambda_n(G)<c'n$ or $\lambda_2(G)>c'n$, for $K_{r+1}$-free graph with sufficiently large order $n$ and size $\lceil cn^2\rceil$, where $c'$ is a positive constant related to fixed $c\in(0,\frac{1}{2})$ and $r$.
More refinements were obtained by Bollob\'{a}s and Nikiforov \cite{BN04}, and Nikiforov \cite{N06}.
For given family of graphs, a recent work by Brooks et al. \cite{BGHLL25} determined the maximum value of $\lambda_2(G)$ for outerplanar graphs of sufficiently large order $n$.

Here, we pay attention to the maximum value $\lambda_2(G)$ for an $F$-free graph with given general $F$.
And then we apply the results to families of $K_{r+1}$-free graphs and other investigated graphs.

Let $\rho^*(n,F)$ be the maximum spectral radius of $F$-free graphs on $n\ge n_F$ vertices, and $G^*(n,F)$ be an extremal graph with $\rho\big(G^*(n,F)\big)=\rho^*(n,F)$. 
We mark $n$ is sufficiency large if $n_F$ is $\infty$.
A graph $G$ is $F$-saturated if $G$ is $F$-free but $G+e$ for any $e\notin E(G)$ contains a copy of $F$ as a subgraph.
The max-min Courant-Fischer theorem indicates $G^*(n,F)$ is always $F$-saturated.

\begin{theorem}\label{thm:F-free and odd n}
Let $G$ be an $F$-free connected graph on $n$ vertices.
If $n$ is odd and $n> 2n_F$,
then 
$$\lambda_2(G)\le \rho^*\left(\frac{n-1}{2},F\right),$$
and the equality holds if and only if $G\in \mathcal{I}\big(G^*(\frac{n-1}{2},F),G^*(\frac{n-1}{2},F)\big)$ and $G$ is $F$-free.
\end{theorem}

\begin{proof}
Let $G^*$ be a graph with the maximum $\lambda_2(G)$ among all $F$-free graphs $G$ of odd order $n> 2n_F$.
Then for any $H\in\mathcal{I}\big(G^*(\frac{n-1}{2},F),G^*(\frac{n-1}{2},F)\big)$, we know
    $\lambda_2(G^*)\ge\lambda_2(H)$. 
Applying Lemma \ref{lm:Cauchy interlacing theorem} to the graph $H$ yields
\begin{align}\label{eq:L1-F}
    \lambda_2(G^*)\ge\rho^*\left(\frac{n-1}{2},F\right).
\end{align}

Let $\emph{\textbf{x}}$ be an eigenvector of $G^*$ corresponding to $\lambda_2(G^*)$, and define
\begin{center}
    $V_+=\{v\in V(G^*): \emph{\textbf{x}}_v>0\}$,\ \ $V_-=\{v\in V(G^*): \emph{\textbf{x}}_v<0\}$\ \ and\ \ $V_0=V(G^*)\backslash(V_+\cup V_-)$.
\end{center}

It is clear that $G^*-uv$ is $F$-free for any edge $uv\in E(G^*)$.
We have each edge $uv$ in $E(V_+,V_-)$ is a cut edge. On the contrary that there exists a non-cut edge $uv\in E(V_+,V_-)$, we may have $G^*-uv$ is connected and $\lambda_2(G^*-uv)>\lambda_2(G^*)$ by Theorem \ref{thm:remove edge between N+ and N-}, which contradicts the maximality of $\lambda_2(G^*)$.

Mention that $\rho^*(n,F)$ is strictly increasing on $n$.
We show the following critical claim.

\begin{claim}\label{clm:e(+,-) le 1-F}
$e(V_+,V_-)\le 1$.
\end{claim}
\noindent\textbf{Proof of Claim \ref{clm:e(+,-) le 1-F}.}
Suppose on the contrary that $e(V_+,V_-)\ge 2$.
There are at least three induced subgraphs, say $H_1$, $H_2$ and $H_3$, in distinct connected components of $G-E(V_+,V_-)$ such that $V(H_i)\subseteq V_+$ or $V(H_i)\subseteq V_-$ for $1\le i\le 3$.
And for any vertex $u\in V(H_i)$ and its neighbors $v$ outside $V(H_i)$, it holds $\emph{\textbf{x}}_u\emph{\textbf{x}}_v\le 0$.
By Lemma \ref{lem:connected by negative edges}, we have
$$\lambda_2(G^*)\le \min\{\lambda_1(H_i):i=1,2,3\}.$$
This together with $\min\{|V(H_i)|:i=1,2,3\}\le \frac{n}{3}$ gives that
$\lambda_2(G^*)\le \rho^*\left(\lfloor\frac{n}{3}\rfloor,F\right)$, which contradicts Eq.\,(\ref{eq:L1-F}) since $\rho^*\left(\lfloor\frac{n}{3}\rfloor,F\right)<\rho^*\left(\frac{n-1}{2},F\right)$. $\hfill\blacksquare$ 

By Claim \ref{clm:e(+,-) le 1-F}, we have $e(V_+,V_-)= 0$ or $1$. 
We distinguish the following two cases depending on the value of $e(V_+,V_-)$.

\textbf{Case 1.} $e(V_+,V_-)=0$. Then $V_0\not=\varnothing$ since $G^*$ is connected. 
We assert $|V_0|=1$.
Assume on the contrary that $|V_0|\ge 2$. 
It is clear that $V_0$ is a vertex cut of $G^*$.
Then by Lemma \ref{lem:null cut-set}, we have $\lambda_2(G^*)=\lambda_1(G^*[V_+])=\lambda_1(G^*[V_-])$.
Since $\min\{|V_+|,|V_-|\}\le \lfloor\frac{n-2}{2}\rfloor=\frac{n-1}{2}-1$.
Then
$$\lambda_2(G^*)=\lambda_1(G^*[V_+])=\lambda_1(G^*[V_-])\le \rho^*\left(\frac{n-1}{2}-1,F\right) <\rho^*\left(\frac{n-1}{2},F\right),$$
a contradiction to Eq.\,(\ref{eq:L1-F}).
So $|V_0|=1$. 

Moreover, if $\big||V_+|-|V_-|\big|\ge 2$, we also have $\min\{|V_+|,|V_-|\}\le\frac{n-1}{2}-1$, and then get a contradiction. So $|V_+|=|V_-|$.
Combining the maximum of $\lambda(G^*)$ with Theorem \ref{thm: add edge in N+/N-}, this indicates $G^*\in \mathcal{I}\left(G^*\left(\frac{n-1}{2},F\right),G^*\left(\frac{n-1}{2},F\right)\right)$.

\textbf{Case 2.} $e(V_+,V_-)= 1$. 
We assert $|V_0|=0$. 
Clearly, each vertex in $V_0$ (if any) is adjacent to some vertices in both $V_+$ and $V_-$, and the edge in $E(V_+,V_-)$ is a cut edge. If $V_0\not=\varnothing$, then $G^*$ contains at least 3 connected components, denoted by $H_1$, $H_2$ and $H_3$, induced by $V_+$ and $V_-$. 
Since $\min\{|V(H_i)|:i=1,2,3\}\le \frac{n-|V_0|}{3}\le\frac{n-1}{3}$. Thus,
from Lemmas \ref{lem:null cut-set} and \ref{lem:connected by negative edges}, we get 
$$\lambda_2(G^*)\le\min\{ \lambda_1(H_i):i=1,2,3\}\le\rho^*\left(\left\lfloor\frac{n-1}{3}\right\rfloor,F\right),$$
a contradiction to Eq.\,(\ref{eq:L1-F}).
So $|V_0|=0$.

The graph $G^*$ can be viewed as a graph by adding an edge between $G^*[V_+]$ and $G^*[V_-]$.
Note that $\min\{|V(H_i)|:i=1,2\}\le \lfloor\frac{n}{2}\rfloor=\frac{n-1}{2}.$
By Lemma \ref{lem:connected by negative edges}, we get 
$$\lambda_2(G^*)<\min\big\{\lambda_1(G^*[V_+]),\lambda_1(G^*[V_-])\big\}\le\rho^*\left(\frac{n-1}{2},F\right),$$
a contradiction to Eq.\,(\ref{eq:L1-F}).

These two cases imply that $G^*\in \mathcal{I}\big(G^*(\frac{n-1}{2},F),G^*(\frac{n-1}{2},F)\big)$ and $\lambda_2(G^*)=\rho^*\left(\frac{n-1}{2},F\right)$.
Moreover, we can check that $\lambda_2(G)=\lambda_1\left(G^*(\frac{n-1}{2},F)\right)=\rho^*(\frac{n-1}{2},F)$ for any graph $G$ in $\mathcal{I}\big(G^*(\frac{n-1}{2},F),G^*(\frac{n-1}{2},F)\big)$ by Lemma \ref{lm:Cauchy interlacing theorem}.
\end{proof}

Before describing the graph $G^*$ with maximum the second largest eigenvalue for $F$-free graphs of even order $n$, we need to introduce a proper lower bound for $\lambda_2(G^*)$.

\begin{lemma}\label{prop:asymptotic value for F}
If $G^*$ is a graph with the maximum $\lambda_2(G)$ among $F$-free connected graphs $G$ of even order $n\ge 2n_F$, where $F$ does not contain cut edges. Then 
$$\lambda_2(G^*)>\rho^*\left(\frac{n}{2},F\right)-\frac{2}{n}.$$
\end{lemma}
\begin{proof}
Let $\textit{\textbf{x}}$ be the Perron vector of $G^*(\frac{n}{2},F)$.
Define the space $\mathbb{W}$ as
\begin{center}
    $\mathbb{W}=\mathrm{span}\left(\left[
    \begin{matrix}
    \textit{\textbf{x}}\\
    \textit{\textbf{x}}
    \end{matrix}
    \right],
    \left[
    \begin{matrix}
    \textit{\textbf{x}}\\
    -\textit{\textbf{x}}
    \end{matrix}
    \right]\right)$.
\end{center}
It is clear that $\left[
    \begin{matrix}
    \textit{\textbf{x}}\\
    \textit{\textbf{x}}
    \end{matrix}
    \right]\perp\left[
    \begin{matrix}
    \textit{\textbf{x}}\\
    -\textit{\textbf{x}}
    \end{matrix}
    \right]$, so $\dim(\mathbb{W})=2$.
For any nonzero vector $\textit{\textbf{y}}\in \mathbb{W}$, we have $\textit{\textbf{y}}=s\left[
    \begin{matrix}
    \textit{\textbf{x}}\\
    \textit{\textbf{x}}
    \end{matrix}
    \right]+t\left[
    \begin{matrix}
    \textit{\textbf{x}}\\
    -\textit{\textbf{x}}
    \end{matrix}
    \right]$ with 
\begin{align}\label{eq:norm of y-F}
    \|\textit{\textbf{y}}\|^2=(s+t)^2\|\textit{\textbf{x}}\|^2+(s-t)^2\|\textit{\textbf{x}}\|^2=2(s^2+t^2),
\end{align}
where $s$ and $t$ are not both zero.

Let $u$ (or $v$, respectively) be a vertex of $G^*\left(\frac{n}{2},F\right)$ (or a copy of $G^*\left(\frac{n}{2},F\right)$, respectively) with the minimal element in \textit{\textbf{x}}.
Then $\textit{\textbf{x}}_u=\textit{\textbf{x}}_v=\min\left\{\textit{\textbf{x}}_w: w\in V\left(G^*\left(\frac{n}{2},F\right)\right)\right\}$, which indicates 
$\frac{n}{2}\textit{\textbf{x}}_u^2\le\sum_{w\in V\left(G^*\left(\frac{n}{2},F\right)\right)}\textit{\textbf{x}}_w^2=1$, that is,
\begin{align}\label{eq:min x-w-F}
    \textit{\textbf{x}}_u^2\le\frac{2}{n}.
\end{align}

The maximality of $\lambda_2(G^*)$ gives $\lambda_2(G^*)\ge\lambda_2\Big(G^*\big(\frac{n}{2},F\big)uvG^*\big(\frac{n}{2},F\big)\Big)$. 
It suffices to show $\lambda_2\Big(G^*\big(\frac{n}{2},F\big)uvG^*\big(\frac{n}{2},F\big)\Big)>\rho^*\left(\frac{n}{2},F\right)-\frac{2}{n}$.

Let $\textbf{e}_u$ be the normal $\frac{n}{2}$-vector with $u$th element $1$.
We get
\begin{align*}
    &\textit{\textbf{y}}^TA\Big(G^*\big(\frac{n}{2},F\big)uvG^*\big(\frac{n}{2},F\big)\Big)\textit{\textbf{y}}\\
    &=\left[(s+t)\textit{\textbf{x}}^T\ \  (s-t)\textit{\textbf{x}}^T\right]\left[
    \begin{matrix}
    A\Big(G^*\big(\frac{n}{2},F\big)\Big) & \textbf{e}_u\textbf{e}_v^T\\
    \textbf{e}_v\textbf{e}_u^T & A\Big(G^*\big(\frac{n}{2},F\big)\Big)\\
    \end{matrix}
    \right]
    \left[
    \begin{matrix}
    (s+t)\textit{\textbf{x}}\\
    (s-t)\textit{\textbf{x}}\\
    \end{matrix}
    \right]\\
    &=(s+t)^2\textit{\textbf{x}}^TA\Big(G^*\big(\frac{n}{2},F\big)\Big)\textit{\textbf{x}}+(s-t)^2\textit{\textbf{x}}^TA\Big(G^*\big(\frac{n}{2},F\big)\Big)\textit{\textbf{x}}+2(s+t)(s-t)\textit{\textbf{x}}_u\textit{\textbf{x}}_v\\
    &=2(s^2+t^2)\textit{\textbf{x}}^TA\Big(G^*\big(\frac{n}{2},F\big)\Big)\textit{\textbf{x}}+2(s^2-t^2)\textit{\textbf{x}}_u^2\\
    &=2(s^2+t^2)\lambda_1\Big(G^*\big(\frac{n}{2},F\big)\Big)+2(s^2-t^2)\textit{\textbf{x}}_u^2\\
    &=2(s^2+t^2)\rho^*\left(\frac{n}{2},F\right)+2(s^2-t^2)\textit{\textbf{x}}_u^2,
\end{align*}
which combining with Eqs.\,(\ref{eq:norm of y-F}) and (\ref{eq:min x-w-F}) follows that
\begin{align*}
    \frac{\textit{\textbf{y}}^TA\Big(G^*\big(\frac{n}{2},F\big)uvG^*\big(\frac{n}{2},F\big)\Big)\textit{\textbf{y}}}{\|\textit{\textbf{y}}\|^2}&=\rho^*\left(\frac{n}{2},F\right)+\frac{2(s^2-t^2)\textit{\textbf{x}}_u^2}{2(s^2+t^2)}\\
    &=\rho^*\left(\frac{n}{2},F\right)+\left(\frac{2s^2}{s^2+t^2}-1\right)\textit{\textbf{x}}_u^2\\
    &\ge \rho^*\left(\frac{n}{2},F\right)-\textit{\textbf{x}}_u^2\\
    &\ge \rho^*\left(\frac{n}{2},F\right)-\frac{2}{n}.
\end{align*}
By Eq.\,(\ref{eq:relation of lambda 2}), we have
\begin{align*}
    \lambda_2\Big(G^*\big(\frac{n}{2},F\big)uvG^*\big(\frac{n}{2},F\big)\Big)\ge\min_{\textit{\textbf{y}}\in\mathbb{W}}{\frac{\textit{\textbf{y}}^TA\Big(G^*\big(\frac{n}{2},F\big)uvG^*\big(\frac{n}{2},F\big)\Big)\textit{\textbf{y}}}{\|\textit{\textbf{y}}\|^2}}\ge \rho^*\left(\frac{n}{2},F\right)-\frac{2}{n}.
\end{align*}
If the above equality holds, then $s=0$ and $\textit{\textbf{x}}_u^2=\frac{2}{n}$.
This implies that $\textit{\textbf{y}}=\frac{1}{\sqrt{n}}\left[
    \begin{matrix}
    \textbf{1}\\
    -\textbf{1}
    \end{matrix}
    \right]$ (let $t=\frac{1}{\sqrt{2}}$) and $G^*\left(\frac{n}{2},F\right)$ is regular.
And then 
\begin{align}\label{eq:for not equal}
\lambda_2\Big(G^*\big(\frac{n}{2},F\big)uvG^*\big(\frac{n}{2},F\big)\Big)=\textit{\textbf{y}}^TA\Big(G^*\big(\frac{n}{2},F\big)uvG^*\big(\frac{n}{2},F\big)\Big)\textit{\textbf{y}}.    
\end{align}
However, it is easy to check that $$A\Big(G^*\big(\frac{n}{2},F\big)uvG^*\big(\frac{n}{2},F\big)\Big)\textit{\textbf{y}}\not=\lambda_2\Big(G^*\big(\frac{n}{2},F\big)uvG^*\big(\frac{n}{2},F\big)\Big)\textit{\textbf{y}}.$$
So $\lambda_2\Big(G^*\big(\frac{n}{2},F\big)uvG^*\big(\frac{n}{2},F\big)\Big)\not=\textit{\textbf{y}}^TA\Big(G^*\big(\frac{n}{2},F\big)uvG^*\big(\frac{n}{2},F\big)\Big)\textit{\textbf{y}}$, a contradiction with Eq.\,(\ref{eq:for not equal}).

Thus, we obtain $\lambda_2(G^*)\ge\lambda_2\Big(G^*\big(\frac{n}{2},F\big)uvG^*\big(\frac{n}{2},F\big)\Big)>\rho^*\left(\frac{n}{2},F\right)-\frac{2}{n}$.
\end{proof}

\begin{theorem}\label{thm:F-free and even n}
Let $G^*$ be a graph with the maximum $\lambda_2(G)$ among $F$-free connected graphs $G$ on $n$ vertices, where $F$ does not contain cut edges. If $n$ is even and $n\ge 2n_F$, then
\begin{center}
  $\lambda_2(G^*)=\rho^*\left(\frac{n}{2},F\right)-\textbf{o}(1)$\ \ and \ \    $G^*\in\mathcal{E}(H_1,H_2)$,
\end{center}
where $H_1$ and $H_2$ are $F$-saturated and $n(H_1)=n(H_2)=\frac{n}{2}$.
\end{theorem}
\begin{proof}
For $i=1,2$, let $G_i$ be a copy of $G^*\left(\frac{n}{2}-1,F\right)$, and $G'_i$ be obtained from $G_i$ by attaching a pendent vertex $u_i$ to some vertex $w_i$ of $G_i$.
Clearly, the graph $G'_1u_1u_2G'_2$ is $F$-free and connected.
Since $G^*$ is a graph with the maximum $\lambda_2(G)$ among $F$-free connected graphs $G$ on $n$ vertices.
Then $\lambda_2(G^*)\ge\lambda_2(G'_1u_1u_2G'_2)$.
From \cite[Theorem \textcolor{blue}{3}]{S74},
the characteristic polynomial of $G'_1u_1u_2G'_2$ satisfies
$$\phi(G'_1u_1u_2G'_2,\lambda)=\phi^2(G'_1,\lambda)-\phi^2(G_1,\lambda)=\Big(\lambda\phi(G_1,\lambda)-\phi(G_1-w_1,\lambda)\Big)^2-\phi^2(G_1,\lambda).$$
Then 
$$\phi\left(G'_1u_1u_2G'_2,\rho^*\big(\frac{n}{2}-1,F\big)\right)=\phi^2\left(G_1-w_1,\rho^*\big(\frac{n}{2}-1,F\big)\right)>0,$$ 
since $G_1-w_1$ is a proper subgraph of $G_1$. 
Moreover, the graph $G'_1$ is a proper subgraph of $G'_1u_1u_2G'_2$, so $\lambda_1(G'_1u_1u_2G'_2)>\rho^*\left(\frac{n}{2}-1,F\right)$.
This deduces
\begin{equation*}
    \lambda_2(G'_1u_1u_2G'_2)>\rho^*\left(\frac{n}{2}-1,F\right).
\end{equation*}
This together with the maximality of $\lambda_2(G^*)$ gives that
\begin{equation}\label{eq:L2-F}
    \lambda_2(G^*)>\rho^*\left(\frac{n}{2}-1,F\right).
\end{equation}

Let $\emph{\textbf{x}}$ be an eigenvector of $G^*$ corresponding to $\lambda_2(G^*)$, and define
\begin{center}
    $V_+=\{v\in V(G^*): \emph{\textbf{x}}_v>0\}$,\ \ $V_-=\{v\in V(G^*): \emph{\textbf{x}}_v<0\}$\ \ and\ \ $V_0=V(G^*)\backslash(V_+\cup V_-)$.
\end{center}

Note $G^*-uv$ is still $F$-free for any edge $uv\in E(V_+,V_-)$.
By Theorem \ref{thm:remove edge between N+ and N-} and the maximality of $\lambda_2(G^*)$, we have each edge
$uv$ in $E(V_+,V_-)$ is a cut edge.
A similar method as the proof of Claim \ref{clm:e(+,-) le 1-F} in Theorem \ref{thm:F-free and odd n} can give that
$e(V_+,V_-)\le 1$.

We now distinguish the following two cases.

\textbf{Case 1.} $e(V_+,V_-)= 0$. Then $|V_0|\ge1$ since $G^*$ is connected.
By Lemma \ref{lem:null cut-set}, we have $\lambda_2(G^*)=\lambda_1(G[V_+])=\lambda_1(G[V_-])$.
Since $V_0$ is a vertex cut and $\min\{|V_+|,|V_-|\}\le \lfloor\frac{n-1}{2}\rfloor=\frac{n}{2}-1$.
Then
$$\lambda_2(G^*)=\lambda_1(G[V_+])=\lambda_1(G[V_-])\le \rho^*\left(\frac{n}{2}-1,F\right),$$
a contradiction to Eq.\,(\ref{eq:L2-F}).

\textbf{Case 2.} $e(V_+,V_-)= 1$. 
A similar process as that of Case 2 in Theorem \ref{thm:F-free and odd n} can give that $|V_0|=0$.
So the graph $G^*$ can be viewed as a graph obtained from $G^*[V_+]$ and $G^*[V_-]$ by adding an edge between them.
If $\big||V_+|-|V_-|\big|\ge 2$, then $\min\{|V_+|,|V_-|\}\le \frac{n}{2}-1$. 
By Lemma \ref{lem:connected by negative edges}, we get 
$$\lambda_2(G^*)<\min\big\{\lambda_1(G^*[V_+]),\lambda_1(G^*[V_-])\big\}\le\rho^*\left(\frac{n}{2}-1,F\right),$$
a contradiction to Eq.\,(\ref{eq:L2-F}). So $|V_+|=|V_-|=\frac{n}{2}$. 
By Theorem \ref{thm: add edge in N+/N-}, we know both $G^*[V_+]$ and $G^*[V_-]$ are $F$-saturated. 

Thus $G^*\in\mathcal{E}(H_1,H_2)$ such that $H_1$ and $H_2$ are $F$-saturated and $n(H_1)=n(H_2)=\frac{n}{2}$.
Meanwhile, we know $V_+=V(H_1)$ and $V_-=V(H_2)$ (without less of generality).
By Lemma \ref{lem:connected by negative edges}, we have 
$$\lambda_2(G^*)<\min\big\{\lambda_1(H_1),\lambda_1(H_2)\big\}\le\rho^*\left(\frac{n}{2},F\right).$$
Combining with Lemma \ref{prop:asymptotic value for F}, we obtain $\lambda_2(G^*)=\rho^*\left(\frac{n}{2},F\right)-\textit{\textbf{o}}(1)$.

This completes the proof. 
\end{proof}

\begin{remark}\label{rek:broader application}
    The restriction on $F$ to be of no cut-edges is used to construct the graph $H$ in Proof of Theorem \ref{thm:F-free and odd n} and $G_1'u_1u_2G_2'$ in Proof of Theorem \ref{thm:F-free and even n}, such that they not contain $F$ as a subgraph.
    So if $H$ and $G_1'u_1u_2G_2'$ are $F$-free, we may remove this restriction.
    This means Theorems \ref{thm:F-free and odd n} and \ref{thm:F-free and even n} can be applicable to a broader family of graphs.
\end{remark}

From Theorems \ref{thm:F-free and odd n} and \ref{thm:F-free and even n}, we can prove an upper bound of the second largest eigenvalue of general  $F$-free graphs.

\begin{theorem}\label{thm:general upper bound for F}
Let $G$ be an $F$-free graph on $n\ge 2n_F$ vertices, then 
$$\lambda_2(G)\le \rho^*\left(\left\lfloor\frac{n}{2}\right\rfloor,F\right),$$
and the equality holds if and only if $$G\in\mathcal{I}\left(G^*\Big(\frac{n-1}{2},F\Big),G^*\Big(\frac{n-1}{2},F\Big)\right)\bigcup\left\{H\cup G^*\Big(\frac{n-1}{2},F\Big):\rho(H)\ge \rho^*\Big(\frac{n-1}{2},F\Big)\right\}$$ when $n$ is odd, or $G\cong 2G^*\big(\frac{n}{2},F\big)$ when $n$ is even.
\end{theorem}
\begin{proof}
Let $G$ be a graph with the maximum $\lambda_2(G)$ among all $F$-free graphs on $n$ vertices.

If $G$ is connected, then from Theorems \ref{thm:F-free and odd n} and \ref{thm:F-free and even n} we have $\lambda_2(G)\le \rho^*\left(\left\lfloor\frac{n}{2}\right\rfloor,F\right)$, and the equality holds only if $n$ is odd and $G\in\mathcal{I}\big(G^*(\frac{n-1}{2},F),G^*(\frac{n-1}{2},F)\big)$.

If $G$ is disconnected, then let $G$ be $H_1\cup H_2$ such that $H_1$ is connected and $\rho(H_1)\ge\rho(H_2)$.
The spectrum of $G$ is the union of spectrum of $H_1$ and $H_2$.
So 
$$\lambda_2(G)=\max\big\{\lambda_2(H_1),\rho(H_2)\big\}.$$
Suppose that $\rho(H_2)>\rho^*\left(\left\lfloor\frac{n}{2}\right\rfloor,F\right)$.
We get $\rho(H_1)>\rho^*\left(\left\lfloor\frac{n}{2}\right\rfloor,F\right)$. This implies $|V(H_1)|>\left\lfloor\frac{n}{2}\right\rfloor$ and $|V(H_2)|>\left\lfloor\frac{n}{2}\right\rfloor$ by the definition of $\rho^*\left(\left\lfloor\frac{n}{2}\right\rfloor,F\right)$.
Then $|V(H_1)|+|V(H_2)|\ge2\left\lfloor\frac{n}{2}\right\rfloor+2>n$, a contradiction.
Thus, 
\begin{align*}
    \rho(H_2)\le\rho^*\left(\left\lfloor\frac{n}{2}\right\rfloor,F\right).
\end{align*}
Moreover, note $|V(H_1)|= n-|V(H_2)|\le n-1$, then we know from Theorems \ref{thm:F-free and odd n} and \ref{thm:F-free and even n} that
\begin{align}\label{eq:disconnected lambdaH1-F}
    \lambda_2(H_1)\le\rho^*\left(\left\lfloor\frac{n-1}{2}\right\rfloor,F\right)\le\rho^*\left(\left\lfloor\frac{n}{2}\right\rfloor,F\right).
\end{align}

We further show the inequality in Eq.\,(\ref{eq:disconnected lambdaH1-F}) is strict.
Assume that $\lambda_2(H_1)=\rho^*\left(\left\lfloor\frac{n}{2}\right\rfloor,F\right)$. By Eq.\,(\ref{eq:disconnected lambdaH1-F}), we know $n$ is odd and $\lambda_2(H_1)=\rho^*\left(\frac{n-1}{2},F\right)$.
If $|V(H_1)|$ is odd, then by Theorem \ref{thm:F-free and odd n} we have
$\lambda_2(H_1)\le \rho^*\left(\frac{|V(H_1)|-1}{2},F\right)\le \rho^*\left(\frac{n-3}{2},F\right)$, a contradiction.
If $|V(H_1)|$ is even, then by Theorem \ref{thm:F-free and even n} we have
$\lambda_2(H_1)< \rho^*\left(\frac{|V(H_1)|}{2},F\right)\le \rho^*\left(\frac{n-1}{2},F\right)$, a contradiction too.
So we obtain $\lambda_2(G)\le \rho^*\left(\left\lfloor\frac{n}{2}\right\rfloor,F\right)$, and the equality holds only if $\lambda_2(G)=\rho(H_2)=\rho^*\left(\left\lfloor\frac{n}{2}\right\rfloor,F\right)$.

When $\lambda_2(G)=\rho(H_2)= \rho^*\left(\left\lfloor\frac{n}{2}\right\rfloor,F\right)$, since $\rho(H_1)\ge\rho(H_2)$, it is easy to see that $|V(H_1)|\ge \left\lfloor\frac{n}{2}\right\rfloor$ and
$|V(H_2)|\ge \left\lfloor\frac{n}{2}\right\rfloor$.
If $n$ is even, then $|V(H_1)|=|V(H_2)|= \frac{n}{2}$. So $H_1=H_2\cong G^*(\frac{n}{2},F)$.
If $n$ is odd, then $\frac{n-1}{2}\le|V(H_1)|\le\frac{n+1}{2}$.
In this situation, we have $H_1\cong G^*(\frac{n-1}{2},F)$ if $|V(H_1)|=\frac{n-1}{2}$, or $H_2\cong G^*(\frac{n-1}{2},F)$ if $|V(H_1)|=\frac{n+1}{2}$ (which gives $|V(H_2)|=\frac{n-1}{2}$).
So $G\in \left\{H\cup G^*(\frac{n-1}{2},F):\rho(H)\ge \rho^*\left(\frac{n-1}{2},F\right)\right\}$.

Furthermore, we can check that $\lambda_2(G)= \rho^*\left(\left\lfloor\frac{n}{2}\right\rfloor,F\right)$ for each graph $G\cong 2G^*(\frac{n}{2},F)$ when $n$ is even, or $G\in\mathcal{I}\big(G^*(\frac{n-1}{2},F),G^*(\frac{n-1}{2},F)\big)\bigcup\big\{H\cup G^*(\frac{n-1}{2},F):\rho(H)\ge \rho^*\left(\frac{n-1}{2},F\right)\big\}$ when $n$ is odd.
\end{proof}

\subsection{Color critical graphs}

A graph $H$ is called to be $\chi$-critical if it has chromatic number $\chi +1$ and contains an edge whose deletion reduces the chromatic number to $\chi$.
In 2009, Nikiforov \cite{N09} gave a spectral version of Simonovits' color critical edge theorem (see \cite{S68}), which determined the graph with maximum number of edges in $H$-free graphs.

Let $k\ge 2$, denote by $T_{n,k}$ the $k$-partite Tur\'an graph on $n$ vertices.

\begin{theorem}\cite{N09}\label{thm:color critical graphs-1}
Let $H$ be a $\chi$-critical graph of order $n_0$.
Then there exists an integer $n_H\ge e^{n_0\chi^{
(2\chi+9)(\chi+1)}}$ such that 
$$\rho^*(n,H)=\rho(T_{n,\chi})$$
and $T_{n,\chi}$ is the unique graph with spectral radius $\rho^*(n,H)$ when $n \ge n_H$.
\end{theorem}

By counting color critical subgraphs, Mubayi \cite{M10} extended Simonovits’ color critical edge theorem.
Recently, Ning and Zhai \cite{NZ23} showed a spectral version of Mubayi's result, surely strengthening Nikiforov's result.
Here, we can generalize Nikiforov's result on the spectral radius to the second largest eigenvalue.

It is easy to note the resulting graph $G$ by adding a pendent vertex to any vertex of $T_{\lfloor\frac{n-1}{2}\rfloor,\chi}$ is still of chromatic number $\chi$, so $G$ is $H$-free. 
Theorems \ref{thm:F-free and odd n} and \ref{thm:F-free and even n}, together with Remark \ref{rek:broader application}, could give the extremal result on $\lambda_2(G)$ among all $\chi$-critical graphs $G$.

\begin{theorem}\label{thm:color critical graphs-2}
Let $H$ be a $\chi$-critical graph and $G$ be an $H$-free connected graph on $n\ge 2n_H$ vertices, then 
(1) when $n$ is odd, $$\lambda_2(G)\le \rho^*\left(\frac{n-1}{2},H\right),$$
with the equality holds if and only if $G\in \mathcal{I}\big(T_{\frac{n-1}{2},\chi},T_{\frac{n-1}{2},\chi}\big)$ and $G$ is $F$-free;
(2) when $n$ is even, 
\begin{center}
  $\lambda_2(G^*)=\rho\left(T_{\frac{n-1}{2},\chi}\right)-\textbf{o}(1)$\ \ and \ \    $G^*\in\mathcal{E}(H_1,H_2)$,
\end{center}
where $H_1$ and $H_2$ are $H$-saturated and $n(H_1)=n(H_2)=\frac{n}{2}$.
\end{theorem}

When the graph $H$ is given in detail, we may determine the value $n_H$, and even the exact $\lambda_2(G^*)$ and the structure of $G^*$ for even $n$.
Here we consider some color critical graphs of great attention, such as complete graph $K_{r+1}$, Book graph $B_{k+1}$ and odd cycle $C_{k+1}$.

\vskip0.2in
\noindent\textit{1. Complete graphs}

Neumaier \cite{N82} and Shao \cite{S95} determined the maximum of the second eigenvalue for tress with odd vertices and even vertices, respectively. 
Powers \cite{P88} extended the results from trees to bipartite graphs, and obtained  
\begin{align*}
    \lambda_2(G)\le\begin{cases}
    k,& \textrm{if}\ n=4k\ \textrm{or}\ 4k+1;\\
    \sqrt{k(k+1)},& \textrm{if}\ n=4k+2\ \textrm{or}\ 4k+3,
    \end{cases}
\end{align*}
where $n$ is the order of $G$ and $k=\left\lfloor\frac{n}{4}\right\rfloor$.

For general graphs, Hong \cite{H88} gave the upper bounds.

\begin{theorem}\cite{H88}\label{thm:general graph}
Let $G$ be a graph with $n$ vertices. Then $\lambda_2(G)\le \frac{n-2}{2}$. 
Equality holds if and only if $G=K_{\frac{n}{2}}\bigcup K_{\frac{n}{2}}$.
Furthermore, if $G$ is a connected graph, then $\lambda_2(G)\le\frac{\sqrt{n^2-4}}{2}-1$.
\end{theorem}

These bounds are not always attainable for connected graphs.
Using graph partitioning, Zhai et al. \cite{ZLW12} reproved Powers and Hong's bounds for bipartite graphs and connected graphs, and completely characterized extremal graphs attaining these bounds.

Note that a graph containing $K_{r+1}$ as a subgraph has chromatic number $\chi(G)\ge r+1$.
We show the maximum value of the second largest eigenvalue for $K_{r+1}$-free graphs. 
Therefore it gives the maximum $\lambda_2(G)$ for 
$\chi$-part graphs $G$, encompassing the previous results in connected graphs and connected bipartite graphs.

We introduce the following result obtained by Nikiforov \cite{N02,N07}, which proved the maximum spectral radius among $K_{r+1}$-free graphs.

\begin{theorem}\cite{N02, N07}\label{thm:clique-1}
If $G$ is a $K_{r+1}$-free graph of order $n\ge r$, then $$\rho(G)<\rho(T_{n,r}),$$ unless $G=T_{n,r}$.
\end{theorem}

The extremal graph $T_{n,r}$ with spectral radius $\rho^*(n,K_{r+1})$ from Theorem \ref{thm:clique-1} is $r$-partite.
Refining Nikiforov's theorem, Li and Peng \cite{LP22} determined the maximum spectral radius of graphs among non-$r$-partite $K_{r+1}$-free graphs.

Let $U_1,\ldots, U_r$ be color sets of $T_{n,r}$ such that $|U_1|\le |U_2|\le\cdots\le |U_r|$.
Let $v\in U_1$ and $u,w\in U_n$, denoted by $Y_{n,r}$ the graph obtained from $T_{n,r}$ by embedding an edge $uv$, and removing $vw$ and $uv'$ for each vertex $v'\in U_1\backslash\{v\}$.
That is $Y_{n,r}=T_{n,r}+uw-wv-\sum_{v'\in U_n\backslash\{v\}}uv'$.

\begin{theorem}\cite{LP22}\label{thm:clique-non-r-partite-1}
Let $G$ be a non-$r$-partite $K_{r+1}$-free graph of order $n$. Then
$$\rho(G)\le\rho(Y_{n,r})$$
with the equality holds if and only if $G =Y_{n,r}$.
\end{theorem}

Before giving the extremal results on the second largest eigenvalue of $K_{r+1}$-free graphs, we need two necessary lemmas.

\begin{lemma}\cite{GWL19}\label{lem:remove pendant vertex-1}
    Let $G$ be a connected graph and $v$ be a pendant vertex of $G$. Then
    $$\rho^2(G)\le\rho^2(G-v)+1.$$
\end{lemma}

\begin{lemma}\label{lem:complete multipartite graph for even n}
Suppose $H_1$ and $H_2$ are complete $r$-partite graphs on $\frac{n}{2}$ where $n$ is even and $r\ge 2$. 
If $G\in \mathcal{E}(H_1,H_2)$, then
$$\lambda_2(G)\le \lambda_2(T_{\frac{n}{2},r}uvT_{\frac{n}{2},r}),$$
where $u$ (and $v$) is a vertex of $T_{\frac{n}{2},r}$ with the minimum degree.
Equality holds if and only if $G=T_{\frac{n}{2},r}uvT_{\frac{n}{2},r}$.
\end{lemma}
\begin{proof}
It is concluded in \cite[p.74]{CDS80} that the characteristic polynomial of a complete multipartite graph $K_{n_1,n_2,\ldots,n_r}$ is
$$\phi_{K_{n_1,n_2,\ldots,n_r}}(x)=x^{n_1+\cdots+n_r-r}\left(x^r-\sum_{s=2}^r(s-1)k_s x^{r-s}\right),$$
where $k_s$ denotes the number of copies of $K_s$.
A theorem by Zykov \cite{Z49} shows $k_s(G)<k_s(T_{|V(G)|,r})$ for $2\le s\le r$ if the clique number of $G$ is $r$. 
Thus, when $x>0$, we have
\begin{align}\label{eq:compare polynomial}
    \phi_{K_{n_1,n_2,\ldots,n_r}}(x)>\phi_{T_{n_1+n_2+\cdots+n_r,r}}(x)
\end{align}
unless $|n_i-n_j|\le 1$ for all $1\le i, j\le r$.

Let $G$ be a graph having the largest second eigenvalue in $\mathcal{E}(H_1,H_2)$, and $uv$ be the edge in $G$ such that $u\in V(H_1)$ and $v\in V(H_2)$.
Since $n$ is even, we get $V(H_1)=N_+$ and $V(H_2)=N_-$ by setting $F$ as $\{K_t:t\ge r+1\}$ in Theorem \ref{thm:F-free and even n}.
So by Lemma \ref{lem:connected by negative edges} \begin{align}\label{eq:compare1}
    \min\big\{\lambda_1(H_1),\lambda_1(H_2)\big\}>\lambda_2(H_1uvH_2).
\end{align}
Moreover, from Cauchy interlacing theorem, we have
\begin{align}\label{eq:compare2}
    \max\big\{\lambda_1(H_1-u),\lambda_1(H_2-v)\big\}\le\lambda_2(H_1uvH_2).
\end{align}

Suppose that $H_1$ is not the Tur\'an graph $T_{\frac{n}{2},r}$.
Let $w$ is a vertex with the minimum degree in $T_{\frac{n}{2},r}$.
Then we compare the characteristic polynomials, by a formula in \cite[Theorem \textcolor{blue}{3}]{S74},
$$\phi_{H_1uvH_2}(x)-\phi_{T_{\frac{n}{2},r}wvH_2}(x)=\Big(\phi_{H_1}(x)-\phi_{T_{\frac{n}{2},r}}(x)\Big)\phi_{H_2}(x)-\Big(\phi_{H_1-u}(x)-\phi_{T_{\frac{n}{2}-1,r}}(x)\Big)\phi_{H_2-v}(x).$$
By Eqs.\,(\ref{eq:compare1}) and (\ref{eq:compare2}), we get
$$\lambda_2(H)\le\lambda_1(H_2-v)\le\lambda_2(H_1uvH_2)<\lambda_1(H_2),$$
where the first inequality dues to Cauchy interlacing theorem.
Thus,
\begin{center}
    $\phi_{H_2}(\lambda_2(H_1uvH_2))<0$\ \ and \ \ $\phi_{H_2-v}(\lambda_2(H_1uvH_2))>0$.
\end{center}
This combining with Eq.\,(\ref{eq:compare polynomial}) follows
$$\phi_{T_{\frac{n}{2},r}wvH_2}\big(\lambda_2(H_1uvH_2)\big)>0.$$
Note that $\lambda_1(T_{\frac{n}{2},r}wvH_2)>\lambda_1(H_2)>\lambda_2(H_1uvH_2)$.
Then $\lambda_2(H_1uvH_2)<\lambda_2(T_{\frac{n}{2},r}wvH_2)$, contradicting the maximality of $\lambda_2(H_1uvH_2)$.
So $H_1$ is a copy of $T_{\frac{n}{2},r}$.
Similarly, we know $H_2$ is also a copy of $T_{\frac{n}{2},r}$.

Suppose $u$ or $v$ is not a vertex with the minimum degree in $T_{\frac{n}{2},r}$.
Without loss of generality, let $T_{\frac{n}{2},r}-u\not=T_{\frac{n}{2}-1,r}$ and $w$ be a vertex with the minimum degree.
Then 
$$\phi_{T_{\frac{n}{2},r}uvT_{\frac{n}{2},r}}(x)-\phi_{T_{\frac{n}{2},r}wvT_{\frac{n}{2},r}}(x)=-\Big(\phi_{T_{\frac{n}{2},r}-u}(x)-\phi_{T_{\frac{n}{2}-1,r}}(x)\Big)\phi_{T_{\frac{n}{2},r}-v}(x).$$
Since $\phi_{T_{\frac{n}{2},r}-u}(x)-\phi_{T_{\frac{n}{2}-1,r}}(x)>0$ and $\phi_{T_{\frac{n}{2},r}-v}\big(\lambda_2(T_{\frac{n}{2},r}uvT_{\frac{n}{2},r})\big)>0$, we have
$$\phi_{T_{\frac{n}{2},r}wvT_{\frac{n}{2},r}}\big(\lambda_2(T_{\frac{n}{2},r}uvT_{\frac{n}{2},r})\big)>0.$$
So $\lambda_2(T_{\frac{n}{2},r}uvT_{\frac{n}{2},r})<\lambda_2(T_{\frac{n}{2},r}wvT_{\frac{n}{2},r})$, a contradiction.
This shows $u$ (and $v$) is a vertex with the minimum degree.

The proof is complete.
\end{proof}

\begin{theorem}\label{thm:complete graph-2}
Let $r\ge 2$ be an integer and $G$ be a $K_{r+1}$-free connected graph of order $n$. 
We have (1) if $n\ge 2r+1$ is odd, then 
$$\lambda_2(G)\le \rho(T_{\frac{n-1}{2},r}).$$
Equality holds if and only if $G\in \mathcal{I}(T_{\frac{n-1}{2},r},T_{\frac{n-1}{2},r})$ and $G$ is $K_{r+1}$-free;
(2) if $n\ge 10r$ is even, then
$$\lambda_2(G)\le \lambda_2\left(T_{\frac{n}{2},r}uvT_{\frac{n}{2},r}\right),$$
where $u$ (or $v$) is a vertex of $T_{\frac{n}{2},r}$ in the partition set of $\lceil\frac{n}{2r}\rceil$ vertices.
Equality holds if and only if $G=T_{\frac{n}{2},r}uvT_{\frac{n}{2},r}$.
\end{theorem}
\begin{proof}
When $n$ is odd, from Theorems \ref{thm:F-free and odd n} and \ref{thm:clique-1}, we can directly get $\lambda_2(G)\le \rho(T_{\frac{n-1}{2},r})$, and the equality holds if and only if $G\in \mathcal{I}(T_{\frac{n-1}{2},r},T_{\frac{n-1}{2},r})$ and $G$ is $K_{r+1}$-free.

When $n$ is even, let $G$ be the graph with the maximum second largest eigenvalue among $K_{r+1}$-free connected graph of order $n$.
Then 
$$\lambda_2(G)\ge \lambda_2(T_{\frac{n}{2},r}uvT_{\frac{n}{2},r}).$$
From Theorem \ref{thm:F-free and even n}, we have $G\in\mathcal{E}(H_1,H_2)$ such that $H_1$ and $H_2$ are $K_{r+1}$-saturated and $n(H_1)=n(H_2)=\frac{n}{2}$.

If $G$ is $r$-partite, then $H_1$ and $H_2$ are complete $r$-partite graphs.
Then from Theorem \ref{lem:complete multipartite graph for even n} we obtain $G$ is isomorphic to $T_{\frac{n}{2},r}uvT_{\frac{n}{2},r}$.

If $G$ is non-$r$-partite, we shall deduce a contradiction.
It is known from the proof of Theorem \ref{thm:F-free and even n} that $$\lambda_2(G)<\min\{\rho(H_1),\rho(H_2)\}.$$
At least one graph of $H_1$ and $H_2$ is non-$r$-partite, so we may let $H_1$ is non-$r$-partite and $\lambda_2(G)<\rho(H_2)$.
From Theorem \ref{thm:clique-non-r-partite-1}, we further get $\lambda_2(G)<\rho(H_2)\le\rho(Y_{\frac{n}{2},r})$.

On the other hand, we get
$$\lambda_2(G)\ge\lambda_2(T_{\frac{n}{2},r}uvT_{\frac{n}{2},r})>\rho(T_{\frac{n}{2},r})-\frac{2}{n}$$
from Lemma \ref{prop:asymptotic value for F}.  
To obtain the contradiction, it suffices to prove
$$\rho(Y_{\frac{n}{2},r})\le\rho(T_{\frac{n}{2},r})-\frac{2}{n}.$$

\begin{claim}\label{clm:gap-1}
  If $n\ge 5r$, then $\rho(Y_{n,r})<\rho(T_{n,r})-\frac{1}{5r}$.
\end{claim}
\noindent\textbf{Proof of Claim \ref{clm:gap-1}.} 
If $r=2$, then it is clear that
\begin{center}
    $\rho(T_{n,2})=\sqrt{\lfloor\frac{n^2}{4}\rfloor}$\ \ and\ \ $\rho(T_{n-1,2})=\sqrt{\lfloor\frac{(n-1)^2}{4}\rfloor}$.
\end{center}
From Lemma \ref{lem:remove pendant vertex-1}, we get
$$\rho(Y_{n,2})\le\sqrt{\rho^2(T_{n-1,2})+1}=\sqrt{\big\lfloor\frac{(n-1)^2}{4}\big\rfloor+1}.$$
Thus, we can estimates the difference 
\begin{align*}
    \rho(T_{n,2})-\rho(Y_{n,2})&\ge\sqrt{\lfloor\frac{n^2}{4}\rfloor}-\sqrt{\big\lfloor\frac{(n-1)^2}{4}\big\rfloor+1}\\
    &=\frac{\lfloor\frac{n^2}{4}\rfloor-\big\lfloor\frac{(n-1)^2}{4}\big\rfloor-1}{\sqrt{\lfloor\frac{n^2}{4}\rfloor}+\sqrt{\big\lfloor\frac{(n-1)^2}{4}\big\rfloor+1}}\\
    &\ge\frac{\frac{n^2}{4}-\frac{(n-1)^2}{4}-2}{n}\\
    &=\frac{1}{2}-\frac{9}{4n}\\
    &>\frac{1}{10}.
\end{align*}

Now suppose $r\ge 3$.
Let $\textbf{\textit{x}}$ be the Perron vector of $A(Y_{n,r})$.
We follow the definition of $Y_{n,r}$, and
let $w_1$ be a vertex in $U_r\backslash\{u,w\}$. 
By the characteristic equation $A(Y_{n,r})\textbf{\textit{x}}=\rho(Y_{n,r})\textbf{\textit{x}}$, we have
\begin{center}
$\rho(Y_{n,r})\textbf{\textit{x}}_{w_1}=\textbf{\textit{x}}_v+\sum_{z\in U_1\backslash\{v\}}\textbf{\textit{x}}_z+Z_0$\ \ and\ \  $\rho(Y_{n,r})\textbf{\textit{x}}_w=\textbf{\textit{x}}_u+\sum_{z\in U_1\backslash\{v\}}\textbf{\textit{x}}_z+Z_0$,
\end{center}
where $Z_0=\sum_{z\in U_2\cup\cdots\cup U_{r-1}}\textbf{\textit{x}}_z$,
which follows
\begin{align}\label{eq:compare component1}
    \rho(Y_{n,r})(\textbf{\textit{x}}_{w_1}-\textbf{\textit{x}}_{w})=\textbf{\textit{x}}_{v}-\textbf{\textit{x}}_{u}.
\end{align}
We also have
\begin{center}
$\rho(Y_{n,r})\textbf{\textit{x}}_v=\textbf{\textit{x}}_u+\sum_{z\in U_r\backslash\{u,w\}}\textbf{\textit{x}}_z+Z_0$\ \ and\ \  $\rho(Y_{n,r})\textbf{\textit{x}}_u=\textbf{\textit{x}}_v+\textbf{\textit{x}}_w+Z_0$,
\end{center}
which, noting $n\ge 5r$, gives 
\begin{align}\label{eq:compare component2}
\big(\rho(Y_{n,r})+1\big)(\textbf{\textit{x}}_v-\textbf{\textit{x}}_u)=\sum_{z\in U_r\backslash\{u,w\}}\textbf{\textit{x}}_z-\textbf{\textit{x}}_w> \textbf{\textit{x}}_{w_1}-\textbf{\textit{x}}_w.
\end{align}
Then, combining Eqs.\,(\ref{eq:compare component1}) and (\ref{eq:compare component2}), we get
$\Big(\rho(Y_{n,r})+1-\frac{1}{\rho(Y_{n,r})}\Big)(\textbf{\textit{x}}_{v}-\textbf{\textit{x}}_{u})> 0$, and so
$\textbf{\textit{x}}_{v}>\textbf{\textit{x}}_{u}.$

Let $Y$ be the graph obtained from $Y_{n,r}$ by removing $uw$ and adding $vw$.
Then, from Theorem \ref{thm:C-F thm}, we have $$\rho(Y)-\rho(Y_{n,r})\ge \textbf{\textit{x}}^T\big(A(Y)-A(Y_{n,r})\big)\textbf{\textit{x}}=2(\textbf{\textit{x}}_v-\textbf{\textit{x}}_u)\textbf{\textit{x}}_w> 0.$$
It is noted that the graph $Y$ is a proper subgraph of $T_{n,r}$, that is, $Y=T_{n,r}-E(u,U_1\backslash\{v\})$.
Let $\textbf{\textit{y}}$ be the Perron vector of $Y$.
Since $A(Y)\textbf{\textit{y}}=\rho(Y)\textbf{\textit{y}}$, we have $\rho(Y)\textbf{\textit{y}}_u=\textbf{\textit{y}}_v+Z_1$,
where $Z_1=\sum_{z\in U_2\cup\cdots\cup U_{r-1}}\textbf{\textit{y}}_z$.
Then $\textbf{\textit{y}}_u=\frac{\textbf{\textit{y}}_v+Z_1}{\rho(Y)}$.
From Theorem \ref{thm:C-F thm}, we get
\begin{align}\label{eq:Tnr-Y(1)}
    \rho(T_{n,r})-\rho(Y)&\ge\textbf{\textit{y}}^T\big(A(T_{n,r})-A(Y)\big)\textbf{\textit{y}}\nonumber\\
    &=2\textbf{\textit{y}}_u\sum_{z\in U_1\backslash\{v\}}\textbf{\textit{y}}_z\nonumber\\
&=\frac{2\left(\textbf{\textit{y}}_v+Z_1\right)\cdot\sum_{z\in U_1\backslash\{v\}}{\textbf{\textit{y}}_z}}{\rho(Y)}.
\end{align}
Since $A(Y)\textbf{\textit{y}}=\rho(Y)\textbf{\textit{y}}$, we also have
\begin{center}    $\rho(Y)\textbf{\textit{y}}_v=\textbf{\textit{y}}_u+\sum_{z\in U_r\backslash\{u\}}{\textbf{\textit{y}}_z}+Z_1$ and $\rho(Y)\sum_{z\in U_1\backslash\{v\}}{\textbf{\textit{y}}_z}=\left(\lfloor\frac{n}{r}\rfloor-1\right)\left(\sum_{z\in U_r\backslash\{u\}}{\textbf{\textit{y}}_z}+Z_1\right)$.
\end{center}
Summing up these two equalities, we get
$\rho(Y)\sum_{z\in U_1}{\textbf{\textit{y}}_z}=\textbf{\textit{y}}_u+\lfloor\frac{n}{r}\rfloor\left(\sum_{z\in U_r\backslash\{u\}}{\textbf{\textit{y}}_z}+Z_1\right)$.
Then
\begin{align*}
    \frac{\sum_{z\in U_1}{\textbf{\textit{y}}_z}}{\sum_{z\in U_1\backslash\{v\}}{\textbf{\textit{y}}_z}}&=\frac{\textbf{\textit{y}}_u+\lfloor\frac{n}{r}\rfloor\left(\sum_{z\in U_r\backslash\{u\}}{\textbf{\textit{y}}_z}+Z_1\right)}{\left(\lfloor\frac{n}{r}\rfloor-1\right)\left(\sum_{z\in U_r\backslash\{u\}}{\textbf{\textit{y}}_z}+Z_1\right)}\\
    &=\frac{\textbf{\textit{y}}_u}{\left(\lfloor\frac{n}{r}\rfloor-1\right)\left(\sum_{z\in U_r\backslash\{u\}}{\textbf{\textit{y}}_z}+Z_1\right)}+\frac{1}{\lfloor\frac{n}{r}\rfloor-1}+1\\
    &\le \frac{\textbf{\textit{y}}_u}{\left(\lfloor\frac{n}{r}\rfloor-1\right)\left(n-\lfloor\frac{n}{r}\rfloor-1\right)\textbf{\textit{y}}_u}+\frac{1}{\lfloor\frac{n}{r}\rfloor-1}+1\\
    &=\frac{1}{\left(\lfloor\frac{n}{r}\rfloor-1\right)\left(1-\frac{1}{n-\lfloor\frac{n}{r}\rfloor}\right)}+1\\
    &\le \frac{4}{3},
\end{align*}
where the first inequality dues to the fact $\textbf{\textit{y}}_u=\min\{\textbf{\textit{y}}_z:z\in V(Y)\}$, and the last inequality dues to $n\ge 5r$ and $r\ge 3$.

By Eq.\,(\ref{eq:Tnr-Y(1)}), we have
\begin{align*}
    \rho(T_{n,r})-\rho(Y)&\ge\frac{2\left(\textbf{\textit{y}}_v+Z_1\right)\cdot\sum_{z\in U_1}{\textbf{\textit{y}}_z}}{\rho(Y)}\cdot \frac{3}{4}\\
    &=\frac{2\left(2\textbf{\textit{y}}_v+2Z_1\right)\cdot\sum_{z\in U_1}{\textbf{\textit{y}}_z}}{\rho(Y)}\cdot \frac{3}{8}\\
    &> \frac{2\left(\textbf{\textit{y}}_u+\sum_{z\in U_r\backslash\{u\}}{\textbf{\textit{y}}_z}+Z_1\right)\cdot\sum_{z\in U_1}{\textbf{\textit{y}}_z}}{\rho(Y)}\cdot \frac{3}{8}\\
    &=\frac{2\left(\sum_{z\in U_r}{\textbf{\textit{y}}_z}+Z_1\right)\cdot\sum_{z\in U_1}{\textbf{\textit{y}}_z}}{\rho(Y)}\cdot \frac{3}{8},
\end{align*}
where the last inequality dues to the fact $\textbf{\textit{y}}_v\ge \textbf{\textit{y}}_u$ and $Z_1\ge \sum_{z\in U_r\backslash\{u\}}{\textbf{\textit{y}}_z}$.
One may see that the resulting graph by removing all edges in $E(U_1,V(Y)\backslash U_1)$ from the graph $Y$ is the union of the Tur\'an graph $T_{n-\lfloor\frac{n}{r}\rfloor,r-1}$ and $\lfloor\frac{n}{r}\rfloor$ copies of $K_1$.
From Theorem \ref{thm:C-F thm}, we obtain
$$\rho(T_{n-\lfloor\frac{n}{r}\rfloor,r-1})-\rho(Y)\ge\textbf{\textit{y}}^T\left(A\big(T_{n-\lfloor\frac{n}{r}\rfloor,r-1}\cup \lfloor\frac{n}{r}\rfloor K_1\big)-A(Y)\right)\textbf{\textit{y}}=-2\left(\sum_{z\in U_r}{\textbf{\textit{y}}_z}+Z_1\right)\cdot\sum_{z\in U_1}{\textbf{\textit{y}}_z}.$$
Thus,
\begin{align}\label{eq:Tnr-Y(2)}
    \rho(T_{n,r})-\rho(Y)\ge \frac{\rho(Y)-\rho(T_{n-\lfloor\frac{n}{r}\rfloor,r-1})}{\rho(Y)}\cdot\frac{3}{8}=\left(1-\frac{\rho(T_{n-\lfloor\frac{n}{r}\rfloor,r-1})}{\rho(Y)}\right)\cdot\frac{3}{8}.
\end{align}

Since the spectral radius of any graph is bounded by its average degree and maximum degree,
we have
$$\rho(Y)\ge\frac{n(n-\lceil\frac{n}{r}\rceil)-2(\lfloor\frac{n}{r}\rfloor-1)}{n}>n-\big\lceil\frac{n}{r}\big\rceil-\big\lfloor\frac{n}{r}\big\rfloor\frac{2}{n}$$
and
$$\rho(T_{n-\lfloor\frac{n}{r}\rfloor,r-1})\le n-2\big\lfloor\frac{n}{r}\big\rfloor.$$
Combining with Eq.\,(\ref{eq:Tnr-Y(2)}) and noting $n\ge 5r$, we get
\begin{align*}
    \rho(T_{n,r})-\rho(Y)&>\frac{2\big\lfloor\frac{n}{r}\big\rfloor-\big\lceil\frac{n}{r}\big\rceil-\big\lfloor\frac{n}{r}\big\rfloor\frac{2}{n}}{n-\big\lceil\frac{n}{r}\big\rceil-\big\lfloor\frac{n}{r}\big\rfloor\frac{2}{n}}\cdot\frac{5}{12}\\
    &>\frac{\frac{n}{r}-2-\big\lfloor\frac{n}{r}\big\rfloor\frac{2}{n}}{n-\big\lceil\frac{n}{r}\big\rceil}\cdot\frac{5}{12}\\
    &\ge \frac{1}{5r}.
\end{align*}
Thus, we obtain
$$\rho(T_{n,r})-\rho(Y_{n,r})>\rho(T_{n,r})-\rho(Y)>\frac{1}{5r}.$$
This claim holds.
$\hfill\blacksquare$ 

From Claim \ref{clm:gap-1}, since $\frac{n}{2}\ge 5r$, we have
$$\rho(T_{\frac{n}{2},r})-\rho(Y_{\frac{n}{2},r})>\frac{1}{5r}\ge\frac{2}{n},$$
completing the proof.
\end{proof}

Theorem \ref{thm:complete graph-2} indicates a directly extension of results of Hong \cite{H88} and Zhai et al. \cite{ZLW12} from connected graphs and connected bipartite graphs to connected $r$-partite graphs.

\begin{corollary}\label{thm:upper bound for connected r-part}
Let $G$ be a connected $r$-partite graph on $n\ge 2r$ vertices, then 
(1) if $n\ge 2r+1$ is odd, then 
$$\lambda_2(G)\le \rho(T_{\frac{n-1}{2},r}).$$
Equality holds if and only if $G\in \mathcal{I}(T_{\frac{n-1}{2},r},T_{\frac{n-1}{2},r})$ and $G$ is $K_{r+1}$-free;
(2) if $n\ge 10r$ is even, then
$$\lambda_2(G)\le \lambda_2\left(T_{\frac{n}{2},r}uvT_{\frac{n}{2},r}\right),$$
where $u$ (or $v$) is a vertex of $T_{\frac{n}{2},r}$ in the partition set of $\lceil\frac{n}{2r}\rceil$ vertices.
Equality holds if and only if $G=T_{\frac{n}{2},r}uvT_{\frac{n}{2},r}$.
\end{corollary}

We can also extend Powers' upper bound of $\lambda_2(G)$ for bipartite graphs to $r$-partite graphs. 

\begin{corollary}\label{thm:upper bound for r-part}
Let $G$ be a $r$-partite graph on $n\ge 2r$ vertices, then 
$$\lambda_2(G)\le \rho(T_{\left\lfloor\frac{n}{2}\right\rfloor,r}).$$
The equality holds if and only if $G\in\mathcal{I}(T_{\frac{n-1}{2},r},T_{\frac{n-1}{2},r})\bigcup\left\{H\cup T_{\frac{n-1}{2},r}:\rho(H)\ge \rho(T_{\frac{n-1}{2},r})\right\}$ when $n$ is odd, or $G\cong 2T_{\frac{n}{2},r}$ when $n$ is even.
\end{corollary}
\begin{proof}
The result follows from Theorem \ref{thm:general upper bound for F}.
\end{proof}


\vskip0.2in
\noindent\textit{2. Book graphs}

As a spectral counterpart of Erd\H{o}s' conjecture \cite{E62} on booksize of graphs, independently resolved by Edwards \cite{E} and Khad\v{z}iivanov and Nikiforov \cite{KN79}, Zhai and Lin \cite{ZL23} determined the spectral condition for existence of the book graph $B_{r+1}$.

\begin{theorem}\cite{ZL23}\label{thm:ZL-book-1}
If $G$ is a $B_{r+1}$-free graph of order $n\ge \frac{13}{2}r$, then
$$\rho(G) \le \rho(T_{n,2}),$$
with equality if and only if $G= T_{n,2}$.
\end{theorem}

Noting the extremal graph $T_{n,2}$ with spectral radius $\rho^*(n,B_{r+1})$ from Theorem \ref{thm:ZL-book-1} is bipartite.
Recently, Liu and Miao \cite{LM25} gave a refinement to Theorem \ref{thm:ZL-book-1} by providing the extremal graph maximizing the spectral radius among non-bipartite $B_{r+1}$-free graphs.

Following the definition in \cite{LM25}, let
$K_{\lfloor\frac{n-1}{2}\rfloor,\lceil\frac{n-1}{2}\rceil}^{r,r}$ be a graph obtained from $T_{n-1,2}$ by adding a new vertex $v$ such that $v$ has $r$ neighbors in each part of $T_{n-1,2}$.

\begin{theorem}\cite{LM25}\label{thm:LM-book-1}
Let $G$ is a non-bipartite $B_{r+1}$-free graph of
order $n\ge8(r^2 + r + 2)$, then
$$\rho(G)\le \rho\left(K_{\lfloor\frac{n-1}{2}\rfloor,\lceil\frac{n-1}{2}\rceil}^{r,r}\right),$$
with equality if and only if $G=K_{\lfloor\frac{n-1}{2}\rfloor,\lceil\frac{n-1}{2}\rceil}^{r,r}$.
\end{theorem}

\begin{theorem}\label{thm:book-2}
Let $r$ be an integer and $G$ be a $B_{r+1}$-free connected graph of order $n$. We have
(1) if $n\ge 13r$ is odd, then
$$\lambda_2(G)\le \rho\left(T_{\frac{n-1}{2},2}\right).$$
Equality holds if and only if $G\in \mathcal{I}(T_{\frac{n-1}{2},2},T_{\frac{n-1}{2},2})$ and $G$ is $B_{r+1}$-free; (2) if $n\ge 16(r^2 + r + 2)$ is even, then
$$\lambda_2(G)\le \lambda_2\left(T_{\frac{n}{2},2}uvT_{\frac{n}{2},2}\right),$$
where $u$ (or $v$) is a vertex of $T_{\frac{n}{2},2}$ in the partition set of $\lceil\frac{n}{4}\rceil$ vertices.
Equality holds if and only if $G=T_{\frac{n}{2},2}uvT_{\frac{n}{2},2}$.
\end{theorem}
\begin{proof}
From Theorems \ref{thm:F-free and odd n} and \ref{thm:ZL-book-1}, the result of (1) holds immediately.

For (2), let $n$ is even and $n\ge 16(r^2 + r + 2)$. 
Then, by Lemma \ref{prop:asymptotic value for F}, it has
\begin{align}\label{eq:book-lower bound}
    \lambda_2\left(G^*(n,B_{r+1})\right)>\rho(T_{\frac{n}{2},2})-\frac{2}{n}.
\end{align}
From Theorem \ref{thm:F-free and even n}, we have the extremal graph $G^*(n,B_{r+1})\in \mathcal{E}(H_1,H_2)$, where $H_1$ and $H_2$ are $B_{r+1}$-saturated and of order $\frac{n}{2}$.
We remark that $\lambda_2(G^*(n,B_{r+1}))<\min\{\rho(H_1),\rho(H_2)\}$ by proof of Theorem \ref{thm:F-free and even n}.
We next characterize the graph $G^*(n,B_{r+1})$.

Suppose that $H_1$ and $H_2$ are bipartite, but $H_1w_1w_2H_2$ is not isomorphic to $T_{\frac{n}{2},2}uvT_{\frac{n}{2},2}$.
It is clear that both $H_1$ and $H_2$ are complete bipartite graphs.
From Lemma \ref{lem:complete multipartite graph for even n}, we get $\lambda_2(H_1w_1w_2H_2)<\lambda_2(T_{\frac{n}{2},2}uvT_{\frac{n}{2},2})$.

Suppose that one of $H_1$ and $H_2$ is non-bipartite.
Then from Theorem \ref{thm:LM-book-1} it has
$$\lambda_2(G^*(n,B_{r+1}))<\min\{\rho(H_1),\rho(H_2)\}\le \rho\left(K_{\lfloor\frac{\frac{n}{2}-1}{2}\rfloor,\lceil\frac{\frac{n}{2}-1}{2}\rceil}^{r,r}\right).$$
By a computational result (see also \cite{LM25}), the spectral radius of $K_{\lfloor\frac{\frac{n}{2}-1}{2}\rfloor,\lceil\frac{\frac{n}{2}-1}{2}\rceil}^{r,r}$, denoted by $\rho_{n,r}$, is the largest root of 
$$f(n,r,x)=x^4-(2r+st)x^2-2r^2x+(2str-sr^2-tr^2),$$
where $s=\left\lfloor\frac{\frac{n}{2}-1}{2}\right\rfloor$ and $t=\left\lceil\frac{\frac{n}{2}-1}{2}\right\rceil$.
When $k\ge 4r^2$, if $n=4k$, then 
$$f\left(4k,r,k-\frac{1}{4}\right)=\frac{1}{2}k^3-\frac{3}{16}k^2+\frac{1}{256}-\left(4k-\frac{3}{2}\right)r^2-\left(k+\frac{1}{8}\right)r>0;$$
and if $n=4k+2$, then
$$f\left(4k+2,r,k+\frac{1}{4}\right)=\frac{1}{2}k^3+\frac{5}{16}k^2+\frac{1}{16}k+\frac{1}{256}-\left(4k+\frac{1}{2}\right)r^2-\left(k+\frac{1}{8}\right)r>0.$$
So $\rho\left(K_{\lfloor\frac{\frac{n}{2}-1}{2}\rfloor,\lceil\frac{\frac{n}{2}-1}{2}\rceil}^{r,r}\right)<\frac{n-1}{4},$
which follows that
$$\lambda_2(G^*(n,B_{r+1}))<\rho\left(K_{\lfloor\frac{\frac{n}{2}-1}{2}\rfloor,\lceil\frac{\frac{n}{2}-1}{2}\rceil}^{r,r}\right)<\frac{n-1}{4}.$$
Moreover, by Eq.\,(\ref{eq:book-lower bound}), we know
\begin{align}\label{eq:lower bound Tn/2,2}
    \lambda_2(G^*(n,B_{r+1}))>\rho(T_{\frac{n}{2},2})-\frac{2}{n}=\sqrt{\left\lfloor\frac{n^2}{16}\right\rfloor}-\frac{2}{n}>\sqrt{\frac{n^2}{16}-1}-\frac{2}{n}>\frac{n}{4}-\frac{6}{n}>\frac{n-1}{4},
\end{align}
a contradiction.

So we obtain that $G^*(n,B_{k+1})$ is $T_{\frac{n}{2},2}uvT_{\frac{n}{2},2}$, where $u$ (or $v$) is a vertex in the partition set of $\lceil\frac{n}{4}\rceil$ vertices.
\end{proof}

\vskip0.2in
\noindent\textit{3. Odd cycles}

In 2008, Nikiforov \cite{N08} gave a spectral condition for the existence of consecutive length of cycles in graphs of sufficiency large order $n$.
Recently, Zou et al \cite{ZFL26} studied the stability of this condition. 
With the help of the result on the stability, we are able to reduce the restriction of $n$ sufficiency large for the existence of consecutive odd length of cycles.
We introduce these two works.

\begin{theorem}\cite{N08}\label{thm:N-odd cycles}
Let $G$ be a $C_{2k+1}$ graph of $n\ge 320(2k+1)$. Then 
$$\rho(G)\le\rho(T_{n,2}).$$ 
\end{theorem}

Let $S_{2k-1}(T_{n-2k+1,2})$ be the graph obtained from $T_{n-2k+1,2}$ by subdividing an edge $2k-1$ times, and $C_{2l+1}(T_{n-2l,2})$ be the graph obtained by identifying a vertex of $C_{2l+1}$ and a vertex of $T_{n-2l,2}$ in a color part of $\left\lfloor
\frac{n-2l}{2}\right\rfloor$ vertices.
For $1\le l\le k$, a $\{C_3, C_5, \ldots, C_{2l-1}, C_{2k+1}\}$-free graph contains no member in $\{C_3, C_5, \ldots, C_{2l-1}, C_{2k+1}\}$ as a subgraph; in particular, it is $C_{2k+1}$-free when $l=1$.

\begin{theorem}\cite{ZFL26}\label{thm:ZFL-odd cycles}
Let $1 \le l \le k$ and $G$ be a $\{C_3, C_5, \ldots, C_{2l-1}, C_{2k+1}\}$-free graph on $n$ vertices, where $n \ge 187k$.\\
(1) If $l=k$ and $\rho(G) \ge \rho(S_{2k-1}(T_{n-2k+1,2}))$, then $G$ is bipartite, unless $G = S_{2k-1}(T_{n-2k+1,2})$.\\
(2) If $l\le k-1$ and $\rho(G) \ge \rho(C_{2l+1}(T_{n-2l,2}))$, then $G$ is bipartite, unless $G = C_{2l+1}(T_{n-2l,2})$.
\end{theorem}

We can give the condition on the second largest eigenvalue for the existence of odd cycles.
Before giving the result, we need to introduce two conclusions about the effects of the spectral radius by graph operations.

\begin{lemma}\cite{WG20}\label{lem:coalescence}
    Let $G$ be the coalescence of two graphs $G_1$ and $G_2$, identifying a vertex of $G_1$ and a vertex of $G_2$, then
    $$\rho^2(G)\le\rho^2(G_1)+\rho^2(G_2),$$
    with the equality holds if and only if $G$ is a star.
\end{lemma}

A walk $v_1v_2\cdots v_k$ ($k \ge 2$) of $G$ is called an internal path, if these $k$ vertices are distinct (except possibly $v_1 = v_k$), $d(v_1) \ge 3$, $d(v_k) \ge 3$ and $d(v_2) = \cdots =
d(v_{k-1}) = 2$ (unless $k = 2$). 
An edge that belongs to some internal path of $G$ is called
an internal edge of $G$. 
Let $W_n$ ($n \ge 6$) be the tree on $n$ vertices obtained from a path $P_{n-4}$ by attaching two new pendant edges to each end vertex of $P_{n-4}$, respectively.

\begin{lemma}\cite{HS75,WG20}\label{lem:subdivision}
Suppose $uv$ is an internal edge not contained in a triangle of a connected graph $G$. 
Let $G_{uv}$ be the graph obtained from $G$ by contracting edge $uv$ (i.e., removing $uv$ and identifying $u$ and $v$ as a single new vertex). 
Then
$$\rho(G)\le\rho(G_{uv}),$$
with the equality holds if and only if $G = W_n$.
\end{lemma}

\begin{theorem}\label{thm:odd cycles-2}
Let $1\le l\le k$ be integers and $G$ be a $\{C_3, C_5, \ldots, C_{2l-1}, C_{2k+1}\}$-free connected graph of order $n\ge 640(2k+1)$. 
We have
(1) if $n$ is odd, then
$$\lambda_2(G)\le \rho\left(T_{\frac{n-1}{2},2}\right).$$
Equality holds if and only if $G\in \mathcal{I}(T_{\frac{n-1}{2},2},T_{\frac{n-1}{2},2})$ and $G$ is $\{C_3, C_5, \ldots, C_{2l-1}, C_{2k+1}\}$-free; (2) if $n$ is even, then
$$\lambda_2(G)\le \lambda_2\left(T_{\frac{n}{2},2}uvT_{\frac{n}{2},2}\right),$$
where $u$ (or $v$) is a vertex of $T_{\frac{n}{2},2}$ in the partition set of $\lceil\frac{n}{4}\rceil$ vertices.
Equality holds if and only if $G=G^*$.
\end{theorem}
\begin{proof}
From Theorems \ref{thm:F-free and odd n} and \ref{thm:N-odd cycles}, the result of (1) holds immediately.

For (2), let $n\ge 640(2k+1)$ be an even integer, and $G^*$ be a graph with the maximum second largest eigenvalue among all $\{C_3, C_5, \ldots, C_{2l-1}, C_{2k+1}\}$-free connected graph of order $n$. 
Then, from Theorem \ref{thm:N-odd cycles} and Lemma \ref{prop:asymptotic value for F}, it has
\begin{align}\label{eq:odd cycles-lower bound}
    \lambda_2\left(G^*\right)>\rho(T_{\frac{n}{2},2})-\frac{2}{n}>\frac{n-1}{4},
\end{align}
where the second inequality dues to Eq.\,(\ref{eq:lower bound Tn/2,2}).
From Theorem \ref{thm:F-free and even n}, we have the extremal graph $G^*\in \mathcal{E}(H_1,H_2)$, where $H_1$ and $H_2$ are $\{C_3, C_5, \ldots, C_{2l-1}, C_{2k+1}\}$-saturated and of order $\frac{n}{2}$.
We remark that $\lambda_2(G^*)<\min\{\rho(H_1),\rho(H_2)\}$ by proof of Theorem \ref{thm:F-free and even n}.
We are ready to characterize the graph $G^*$.

Suppose that $H_1$ and $H_2$ are bipartite, but $G^*$ is not isomorphic to $T_{\frac{n}{2},2}uvT_{\frac{n}{2},2}$.
By a similar method as Proof of Theorem \ref{thm:book-2}(2), we can get a contradiction.

Suppose that at least one of $H_1$ and $H_2$ is non-bipartite.
Then from Theorem \ref{thm:ZFL-odd cycles} it has
$$\lambda_2(G^*)<\min\left\{\rho(H_1),\rho(H_2)\right\}\le 
\begin{cases}
    & \rho\big(S_{2k-1}(T_{\frac{n}{2}-2k+1,2})\big), \hspace{2mm} \mathrm{if}\hspace{2mm} l=k;\\
    & \rho\big(C_{2l+1}(T_{\frac{n}{2}-2l,2})\big), \hspace{7mm} \mathrm{if}\hspace{2mm} l\le k-1.
\end{cases}$$
If $l=k$, then, by Lemma \ref{lem:subdivision}, we have
$$\rho\big(S_{2k-1}(T_{\frac{n}{2}-2k+1,2})\big)<\rho(T_{\frac{n}{2}-2k+1,2})\le\rho(T_{\frac{n}{2}-1,2})=\sqrt{\left\lfloor\frac{(\frac{n}{2}-1)^2}{4}\right\rfloor}\le \frac{n-2}{4}.$$
If $l\le k-1$, then, by Lemma \ref{lem:coalescence}, we have
\begin{align*}
    \rho\left(C_{2l+1}(T_{\frac{n}{2}-2l,2})\right)&\le\sqrt{\rho^2(T_{\frac{n}{2}-2l,2})+\rho^2(C_{2l+1})}\\
    &\le\sqrt{\rho^2(T_{\frac{n}{2}-2,2})+4}\\
    &=\sqrt{\left\lfloor\frac{(\frac{n}{2}-2)^2}{4}\right\rfloor+4}\\
    &\le\sqrt{\frac{n^2-8n+80}{4^2}}\\
    &<\frac{n-3}{4}.
\end{align*}
This indicates that $\lambda_2(G^*)<\frac{n-2}{4}$, a contradiction to Eq.\,(\ref{eq:odd cycles-lower bound}).

So we obtain that $G^*$ is $T_{\frac{n}{2},2}uvT_{\frac{n}{2},2}$, where $u$ (or $v$) is a vertex in the partition set of $\lceil\frac{n}{4}\rceil$ vertices.
\end{proof}

\subsection{Minor-free graphs}

A graph $H$ is said to be a minor of $G$ if $H$ can be obtained from $G$ by a series of edge contractions.
Given a graph family $\mathcal{H}$,
if $G$ not contains any member in $\mathcal{H}$ as a minor, then we call $G$ an $\mathcal{H}$-minor-free graph.

In spectral graph theory, minor-free graphs possess remarkable theoretical properties in their own right.
In this subsection, we focus on the second largest eigenvalue of $K_{r}$- or $K_{s,t}$-minor-free graphs.
Our results are based on some known conclusions maximizing the spectral radius among $K_{r}$- or $K_{s,t}$-minor-free graphs.
Let $G^c$ be the complement of a graph $G$.
Recently, Tait \cite{T19} characterized the $K_r$-minor-free graph of sufficiently large order $n$ with the maximum spectral radius is the join of $K_{r-2}$ and $K^c_{n-r+2}$.
In the same paper \cite{T19}, it is proved the maximum spectral radius of $K_{s,t}$-minor-free graphs with $t\ge s\ge 2$ is obtained by the join of a copy of $K_{s-1}$ and $\frac{n-s+1}{t}$ copies of $K_t$. 
These two works extend existing results in several literature (see \cite{H04, N07, N17}).
Solving a conjecture by Tail, when $t$ is not a factor of $n-s+1$, surprisingly, Zhai and Lin \cite{ZL22} completely determined the graphs with the maximum spectral radius among all $K_{s,t}$-minor-free graphs of any sufficiently large order $n$.

Denote by $G_1\vee G_2$ the join of two disjoint graphs $G_1$ and $G_2$.

\begin{theorem}\cite{T19}\label{thm:Kr minor-1}
Let $r \ge 3$. The $K_r$-minor-free graph of sufficiently large $n$ with
maximum spectral radius is $K_{r-2}\vee K^c_{n-r+2}$.
\end{theorem}

Let $H^\star$ be the Petersen graph, and $H_{s,t}$ be a star forest of order $t + 1$, precisely, the disjoint union of $\big\lfloor\frac{t+1}{s+1}\big\rfloor$ stars in which all but at most one are isomorphic to $K_{1,s}$.
Let $S^1(G)$ be a graph obtained by subdividing the edge with minimum edge degree.

\begin{theorem}\cite{T19, ZL22}\label{thm:Kst minor-1}
For $2 \le s \le t$ and $n - s + 1 = pt + q$, where $1 \le q \le t$ and $n$ is sufficiently,
the graph $K_{s-1}\vee H$ attains the maximum spectral radius among all $K_{s,t}$-minor-free graphs of order $n$, where $H$ is 
\begin{align*}
    H=
    \begin{cases}
        (p-1)K_t\cup \overline{H^\star}, &\ \mathrm{if}\ q=2,\ t=8\ \mathrm{and}\ \big\lfloor\frac{t+1}{s+1}\big\rfloor=1;\\
        (p-1)K_t\cup S^1(\overline{H_{s,t}}),\ &\ \mathrm{if}\ q=\big\lfloor\frac{t+1}{s+1}\big\rfloor=2;\\
        (p-q)K_t\cup q\overline{H_{s,t}},\ &\ \mathrm{if}\ q\le 2\big(\big\lfloor\frac{t+1}{s+1}\big\rfloor-1\big)\ \mathrm{except}\ q=\big\lfloor\frac{t+1}{s+1}\big\rfloor=2;\\
        pK_t\cup K_q,\ &\ \mathrm{otherwise}.
    \end{cases}
\end{align*}
\end{theorem}

So we conclude from Theorems \ref{thm:Kr minor-1} and \ref{thm:Kst minor-1} that $G^*(n,K_r\mathrm{-minor})=K_{r-2}\vee K^c_{n-r+2}$ and $G^*(n,K_{s,t}\mathrm{-minor})=K_{s-1}\vee H$ for $r\ge3$, $t\ge s\ge 2$ and sufficiently large $n$, where $H$ is defined in Theorem \ref{thm:Kst minor-1}.
From Theorems \ref{thm:F-free and odd n} and \ref{thm:F-free and even n}, we can describe graphs with maximum second largest eigenvalue of graphs among $K_r$- or $K_{s,t}$-minor-free graphs.

\begin{theorem}\label{thm:Kr minor-2}
Let $G$ be a graph with maximum second largest eigenvalue among $K_r$-minor-free connected graphs on $n$ vertices, where $r\ge 3$ and $n$ is sufficiently large, then 
\begin{align*}
    G\in
    \begin{cases}
        \mathcal{I}\big(G^*(\frac{n}{2},K_r\mathrm{-minor}),G^*(\frac{n}{2},K_r\mathrm{-minor})\big), &\ \mathrm{when}\ n\ \mathrm{is}\ \mathrm{odd};\\
        \mathcal{E}(H_1,H_2), &\ \mathrm{when}\ n\ \mathrm{is}\ \mathrm{even},
    \end{cases}
\end{align*}
where $H_1$ and $H_2$ are $K_r$-minor-saturated and $n(H_1)=n(H_2)=\frac{n}{2}$.
\end{theorem}

\begin{theorem}\label{thm:Kst minor-2}
Let $G$ be a graph with maximum second largest eigenvalue among $K_{s,t}$-minor-free connected graphs on $n$ vertices, where $t\ge s\ge 2$ and $n$ is sufficiently large, then 
\begin{align*}
    G\in
    \begin{cases}
        \mathcal{I}\big(G^*(\frac{n}{2},K_{s,t}\mathrm{-minor}),G^*(\frac{n}{2},K_{s,t}\mathrm{-minor})\big), &\ \mathrm{when}\ n\ \mathrm{is}\ \mathrm{odd};\\
        \mathcal{E}(H_1,H_2), &\ \mathrm{when}\ n\ \mathrm{is}\ \mathrm{even},
    \end{cases}
\end{align*}
where $H_1$ and $H_2$ are $K_{s,t}$-minor-saturated and $n(H_1)=n(H_2)=\frac{n}{2}$.
\end{theorem}

Kuratowski’s theorem shows a graph is planar if and only if it is $\{K_5,K_{3,3}\}$-minor-free (see \cite{T81}).
A nice theorem by Chartrand and Harary \cite{CH67} shows a graph is outerplanar if and only if it is $\{K_4,K_{2,3}\}$-minor-free.
The study of planar graphs, as well as outerplanar graphs, has a long history. 
Perhaps the most celebrated theorem in graph theory is the Four Color Theorem \cite{AH77, AHK77}, which asserts every planar graph is 4-colorable.

On spectral radius of (outer)planar graphs, it is conjectured by Boots and Royle \cite{BR91} and independently Cao and Vince \cite{CV10} that $G^*(n,\{K_5,K_{3,3}\}\mathrm{-minor})$ is $P_2\vee P_{n-2}$, and conjectured by Cvetkovi\'c and Rowlinson \cite{CR90} that $G^*(n,\{K_4,K_{2,3}\}\mathrm{-minor})$ is $P_1\vee P_{n-1}$.
A recent work by Tait and Tobin \cite{TT17} confirmed Boots-Royle/Vince-Cao conjecture and Cvetkovi\'c-Rowlinson conjecture for sufficiently large order $n$.
It seems to be challenging to solve these conjectures completely.
Fortunately, Lin and Ning \cite{LN21} in 2021 gave a complete solution to Cvetkovi\'c-Rowlinson conjecture for all $n$.

\begin{theorem}\cite{LN21}\label{thm:outerplanar-1}
  Let $n\ge 2$ be an integer, then
  $$G^*(n,\{K_4,K_{2,3}\}\mathrm{-minor})=\begin{cases}
      P_1\vee P_{n-1},&\ \mathrm{when}\ n\not=6;\\
      H,&\ \mathrm{when}\ n=6,
  \end{cases}$$
  where $H$ is obtained from $P_1\vee P_4$ by adding a new vertex $v$ and connecting $u$ with two vertices of degree $3$ in $P_1\vee P_4$.
\end{theorem}

In a recent breakthrough, Brooks, Gu, Hyatt, Linz and Lu \cite{BGHLL25} investigated the second largest eigenvalue of outerplanar graphs.

\begin{theorem}\cite{BGHLL25}\label{thm:outerplanar-2}
  For sufficiently large $n$, let $G$ be a graph with maximum second largest eigenvalue among all outerplanar connected graphs on $n$ vertices, if $n$ is odd, then
  $$G\in \mathcal{I}\big(P_1\vee P_{\frac{n-1}{2}-1},P_1\vee P_{\frac{n-1}{2}-1}\big),$$
  and if $n$ is even, then
  $$G= H_1uvH_2,$$
  where $H_1$ and $H_2$ are copies of $P_1\vee P_{\frac{n}{2}-1}$, and $u$ (and $v$) is a vertex of degree $2$.
\end{theorem}

Here we can extend Theorem \ref{thm:outerplanar-2} by reducing the restriction of order for odd $n$ and describing the graph $G^*(n,\{K_5,K_{3,3}\}\mathrm{-minor})$ for any even $n$.
By Theorems \ref{thm:F-free and odd n} and \ref{thm:F-free and even n}, combining with Theorem \ref{thm:outerplanar-1}, we can directly obtain the following result.

\begin{theorem}
  Let $G$ be a graph with maximum second largest eigenvalue among all outerplanar connected graphs on $n$ vertices, if $n$ is odd, then
  $$
  G\in
  \begin{cases}
  \mathcal{I}\big(P_1\vee P_{\frac{n-1}{2}-1},P_1\vee P_{\frac{n-1}{2}-1}\big),&\ \mathrm{when}\ n\not=13;\\
  \mathcal{I}\big(H,H\big),&\ \mathrm{when}\ n=13,
  \end{cases}
  $$
  where $H$ is defined as Theorem \ref{thm:outerplanar-1},
  and if $n$ is even, then
  $$G\in \mathcal{E}(H_1,H_2)$$
  and $\lambda_2(G)=\rho(P_1\vee P_{\frac{n}{2}-1})-o(1)$, where $H_1$ and $H_2$ are edge-maximal outerplanar graphs on $\frac{n}{2}$ vertices.
\end{theorem}

From the work by Tait and Tobin \cite{TT17}, it proved that for sufficiently large $n$,
$$G^*(n,\{K_5,K_{3,3}\}\mathrm{-minor})=P_2\vee P_{n-2}.$$
Similarly, we can establish an upper bound for $\lambda_2(G)$ among planar graphs.

\begin{theorem}\label{thm:planar-2}
Let $G$ be a graph with maximum second largest eigenvalue among all planar connected graphs on $n$ vertices, where $n$ is sufficiently large, then 
\begin{align*}
    G\in
    \begin{cases}
        \mathcal{I}\big(P_2\vee P_{\frac{n}{2}-2}, P_2\vee P_{\frac{n}{2}-2}\big), &\ \mathrm{when}\ n\ \mathrm{is}\ \mathrm{odd};\\
        \mathcal{E}(H_1,H_2), &\ \mathrm{when}\ n\ \mathrm{is}\ \mathrm{even},
    \end{cases}
\end{align*}
and $\lambda_2(G)=\rho(P_2\vee P_{\frac{n}{2}-2})-o(1)$, where $H_1$ and $H_2$ are edge-maximal planar graphs on $\frac{n}{2}$ vertices.
\end{theorem}





\bibliographystyle{plain}

\end{document}